\documentclass[psamsfonts]{amsart}

\usepackage{amssymb,amsfonts}
\usepackage[all,arc]{xy}
\usepackage{enumerate}
\usepackage{mathrsfs}

\newtheorem{thm}{Theorem}[section]
\newtheorem{cor}[thm]{Corollary}
\newtheorem{prop}[thm]{Proposition}
\newtheorem{lem}[thm]{Lemma}

\theoremstyle{definition}
\newtheorem{defn}[thm]{Definition}

\theoremstyle{remark}
\newtheorem{rem}[thm]{Remark}

\makeatletter
\let\c@equation\c@thm
\makeatother
\numberwithin{equation}{section}

\title{Semigeostrphic equations in physical space with free upper boundary }

\author{Jingrui Cheng}

\begin{document}
\maketitle
\begin{abstract}
We define various notions of Lagrangian solution in physical space for 3-d incompressible semigeostrophic system with free upper boundary under different conditions for initial data, then prove their existence via the minimization with respect to a geostrophic functional, generalizing the the work of \cite{Cullen-Feldman} and \cite{Feldman-Tudorascu} to the case of free upper boundary. As a byproduct of our proof, we obtain the existence of measure-valued dual space solutions when the initial measure $\nu_0\in\mathcal{P}_2(\mathbf{R}^3)$ and is supported on $\{-\frac{1}{\delta}\leq y_3\leq-\delta\}$

\end{abstract}

\tableofcontents

\section{Introduction}
The Semi-Geostrophic system (abbreviated as SG in the following)models large-scale atmospheric-ocean flows, where large scale means the flow is rotation-dominated. J.-M. Benamou and Y.Brenier \cite{Brenier} proved the existence of solutions to incompressible SG system in a fixed domain $\Omega\subset\mathbf{R}^3$ in the so-called dual space formulation, which is a formal change of variable. Under the dual space formulation, the SG system can be written as a transport equation coupled with Monge-Amp$\acute{e}$re equation. Mike Cullen and W.Gangbo \cite{Cullen-Gangbo} considered the free boundary case in 3-D, but with additional assumption that the potential temperature is constant. Under this additional assumption, the system can be rewritten as a system in 2-D, the so-called Semi-geostriophic Shallow Water system. The authors then proved the existence of dual-space solutions with initial dual density in $L^p$($p>1$). The existence of solutions in the original physical variables is first proved by Cullen and Feldman \cite{Cullen-Feldman}, in the Lagrangian formulation of the physical system, for both fixed boundary SG system and shallow water system with the same assumption on the dual density as above, and this assumption amounts to some strict convexity condition for the modified pressure. Then in \cite{Feldman-Tudorascu},\cite{Feldman2}, Feldman and Tudorascu put forward a more general notion of physical solutions, to allow for more general initial data, in particular, they proved the existence of general measure valued solutions in dual space.

In this work, we consider the incompressible SG system in a 3-D domain with free upper boundary, but without the constancy assumption made in \cite{Cullen-Gangbo}. The existence of dual space solutions has been proved in \cite{Pelloni} by invoking the general theory of Hamiltonian ODE established in \cite{AG}. Here we prove the existence of Lagrangian solutions in physical space, generalizing the work of \cite{Cullen-Feldman} and \cite{Feldman-Tudorascu}. The main difficulty involved is the more complicated geostrophic energy in our situation since it involves the unknown free boundary profile. The 3-D SG system with free upper boundary has the following form:

\begin{equation}
D_t(\mathbf{u}^g_1,\mathbf{u}_2^g)+(-\mathbf{u}_2,\mathbf{u}_1)+(\partial_{x_1}p,\partial_{x_2}p)=0
\end{equation}
\begin{equation}\nabla\cdot\mathbf{u}=0
\end{equation}
\begin{equation}
D_t\rho=0
\end{equation}
\begin{equation}
\nabla p=(\mathbf{u}_2^g,-\mathbf{u}_1^g,-\rho)
\end{equation}
(1.1)-(1.4) hold in $\Omega_h$, where $D_t=\partial_t+\mathbf{u}\cdot\nabla$, and $\Omega_{h}=\{(x_1,x_2,x_3)\in\Omega_2\times[0,\infty)|0<x_3<h(t,x_1,x_2)\}$

Here $\Omega_2\subset\mathbf{R}^2$ is a bounded convex region. $h(t,x_1,x_2)\geq0$ describes the unknown free upper boundary. In the above $p$ is the pressure, $\mathbf{u}$ is the velocity, and $\rho$ is the density.

Of course we need to prescribe suitable boundary and free boundary conditions. We require that no flow can penetrate the fixed boundary, the pressure at the top is a constant which without loss of generality we take to be zero. This is given by (1.5) and (1.6) below respectively. (1.7) means that no flow penetrates the free boundary, instead the fluid particles move with the flow.
\begin{equation}
\mathbf{u}\cdot\mathbf{n}=0\,\,\,on\,\,\,\partial\Omega_{h}-\{x_3=h\}
\end{equation}
\begin{equation}
p(x_1,x_2,x_3)=0\,\,\,on\,\,\,\{x_3=h\}
\end{equation}
\begin{equation}
\partial_th+\mathbf{u}_1\partial_{x_1}h+\mathbf{u}_2\partial_{x_2}h=\mathbf{u}_3\,\,\,on\,\,\,\{x_3=h\}
\end{equation}

We remark that (1.5) and (1.7) combined express conservation of mass , and are formally equivalent to
$$\partial_t\sigma_h+\nabla\cdot(\mathbf{u}\sigma_h)=0$$
where

\begin{equation}
\sigma_h(x_1,x_2,x_3)=\chi_{\Omega_{h}}(x_1,x_2,x_3)
\end{equation}

Now we put
\begin{equation}
P(t,x)=p(t,x)+\frac{1}{2}(x_1^2+x_2^2)
\end{equation}

then the above system can be written as

\begin{equation}
D_t(\nabla P)=J(\nabla P-x)
\end{equation}
\begin{equation}\nabla\cdot\mathbf{u}=0
\end{equation}

In the above \begin{displaymath}
J =
\left( \begin{array}{ccc}
0 &-1 & 0 \\
1& 0 & 0 \\
0 & 0 &0
\end{array} \right)
\end{displaymath}

The free boundary condition for $P$ is
\begin{equation}
P(x_1,x_2,x_3)=\frac{1}{2}(x_1^2+x_2^2)\,\,\,on\,\,\,\{x_3=h\}
\end{equation}

The geostrophic energy is
\begin{equation}
E=\int_{\Omega_h}\frac{1}{2}((\mathbf{u}_1^{g})^2+(\mathbf{u}_2^{g})^2)+\rho x_3dx=\int_{\Omega_h}\frac{1}{2}[(\partial_1P-x_1)^2+(\partial_2P-x_2)^2]-\partial_3P x_3dx
\end{equation}

By Cullen`s stability principle, the function $P$ above should be convex, and $\nabla P$ should minimize above functional among all possible rearrangement of particles. To make this precise, we are motivated to consider the following minimization problem.

\begin{equation}E_{\nu}(h,\mathbf{T})=\int_{\Omega_{\infty}}\frac{1}{2}[(x_1-\mathbf{T}_1(x))^2+(x_2-\mathbf{T}_2(x))^2]-x_3\mathbf{T}_3(x)dx
\end{equation}

Here $\nu\in\mathcal{P}(\mathbf{R}^3)$ is fixed, $\mathbf{T}_{\sharp}\sigma_h=\nu$ and we require that the actual free boundary profile $h$ and the Borel map $\nabla P$ minimizes $E_{\nu}(h,\mathbf{T})$ among all pairs of $(h,\mathbf{T})$ such that $h\geq0$, continuous, $\int_{\Omega_2}h=1$, and $\mathbf{T}_{\sharp}\sigma_h=\nu$.

Also we recall that $\partial_{x_3}P=-\rho$, where $\rho$ represents the density, so it`s reasonable to assume the convex potential P should satisfy
$-\frac{1}{\delta}\leq\partial_{x_3}P\leq-\delta$, or we require $supp\,\,\nu\subset \mathbf{R}^2\times[-\frac{1}{\delta},-\delta]$.
More generally, we can consider

\begin{equation}
E_{\nu}(h,\gamma)=\int[ \frac{1}{2}[(x_1-y_1)^2+(x_2-y_2)^2]-x_3y_3]d\gamma(x,y)
\end{equation}

Here
 $$h\geq0,\,\,\int_{\Omega_2}hdx_1dx_2=1\,\,\gamma\in\Gamma(\sigma_h,\nu)$$

Here $\Gamma(\sigma_h,\nu)$ denotes the set of measures on the product space $\Omega_{\infty}\times\mathbf{R}^3$ with marginals $\sigma_h$ and $\nu$. Here $\Omega_{\infty}=\Omega_2\times[0,\infty)$. This is the relaxed Kantorovitch problem of (1.14). For a fixed $\sigma_h$, if $\gamma\in \Gamma(\sigma_h,\nu)$ minimize (1.15), then $\gamma$ is called an optimal transport plan of (1.15). If the optimal plan has the form $\gamma=(id\times \mathbf{T})\sharp\sigma_h$ for some Borel map $\mathbf{T}$, then this problem reduces to (1.14).

In the following, we will denote the cost function:
\begin{equation}
c(x,y)=\frac{1}{2}[(x_1-y_1)^2+(x_2-y_2)^2]-x_3y_3=\frac{1}{2}(x_1^2+x_2^2+y_1^2+y_2^2)-x\cdot y
\end{equation}

Here $x=(x_1,x_2,x_3)\in\Omega_{\infty}$ and $y=(y_1,y_2,y_3)\in\mathbf{R}^3$.

 Define $\nu=\nabla P_{\sharp}\sigma_h$, where $P$ is the function appearing in (1.10), then formally it satisfies the following continuity equation. The derivation can be found in \cite{Pelloni}, section 2.

\begin{equation}
\partial_t\nu+\nabla\cdot(\nu\mathbf{w})=0\,\,\,in\,\,[0,T]\times\Lambda
\end{equation}
\begin{equation}
\mathbf{w}(t,y)=J(y-\nabla P^*(t,y))\,\,\,in\,\,[0,T]\times\Lambda
\end{equation}
\begin{equation}
(h,\nabla P)\,\,minimizes\,\,E_{\nu}(\bar{h},\bar{\mathbf{T}})\,\,\,among\,\,\,all\,\,\,pairs\,\,\,such\,\,\,that\,\,\,\bar{\mathbf{T}}_{\sharp}\sigma_{\bar{h}}=\nu
\end{equation}

Pelloni et al proved in \cite{Pelloni} the existence of weak solutions to above SG system in dual variables, (1.17)-(1.19), using the theory of Hamilton ODEs developed in \cite{BV vector} . In this paper, we will reprove this result using a more straightforward time-stepping argument. Then we prove the existence of solutions to physical equations in Lagrangian variables when the dual density $\nu$ defined above is absolutely continuous. Finally we define a notion of relaxed Lagrangian solution similar to \cite{Feldman-Tudorascu} and prove its existence. This allows us to deal with the case where $\nu\in\mathcal{P}_2(\mathbf{R}^3)$(probability measures with finite second order moment).

Here are some notations and terminology which will be used thoughout this work.
In the following, $\Omega_{\infty}$ represents the region $\Omega_2\times[0,\infty)$,and $\Omega_H=\Omega_2\times[0,H]$.\\
Given $h:\Omega_2\rightarrow\mathbf{R}_+$, we denote $\Omega_h=\{x\in\Omega_{\infty}|0<x_3<h(x_1,x_2)\}$, and $\sigma_h(x)=\chi_{\Omega_h}(x)$. We identify an absolutely continuous measure (with respect to $\mathcal{L}^3$) with its densities.

Suppose $A,B\subset\mathbf{R}^3$, and we have two functions $f(x),g(y)$ defined on A and B respectively. We say $f,g$ are convex conjugate to each other over $A$ and $B$ if the following holds.
$$f(x)=\sup_{y\in B}(x\cdot y-g(y))\,\,\,x\in A$$
and
$$g(y)=\sup_{x\in A}(x\cdot y-f(x))\,\,\,y\in B$$

In the sequal, let $p\geq 1$, we denote $\mathcal{P}_p(\mathbf{R}^3)$ to be the set of probability measures in $\mathbf{R}^3$ such that $\int_{\mathbf{R}^3}|y|^pdy<\infty$. For $\nu\in\mathcal{P}_p(\mathbf{R}^3)$, denote $M_p(\nu)=\int|y|^p$

In the following, $\nabla_2$ mans the gradient with respect to the first 2 variables only, namely $\nabla_2v=(\partial_1v,\partial_2v)$

Given $\mu,\nu\in\mathcal{P}_p(\mathbf{R}^3)$($p\geq1$), we define the $p$-Wasserstein metric to be
$$W^p_p(\mu,\nu)=\inf_{\gamma\in\Gamma(\mu,\nu)}\int_{\mathbf{R}^3\times\mathbf{R}^3}|x-y|^pd\gamma(x,y)
$$
Here $\Gamma(\mu,\nu)$ denotes the set of probability measures on $\mathbf{R}^3\times\mathbf{R}^3$ with marginals $\mu,\nu$ respectively. $W_p(\mu,\nu)$ is indeed a metric on $\mathcal{P}_p(\mathbf{R}^3)$, see \cite{gradient-flows} chapter 7.

To conclude this section, we briefly describe the plan of this paper.

In section 2, we study the geostrophic functional and it`s dual problem in detail, and establish various properties of the optimizers which will be used later on. Then in the case when the dual density $\nu\in L^q$, for some $q>1$ with compact support, we follow \cite{Cullen-Feldman} to establish the existence of weak Lagrangian solutions, using the theory of Lagrangian flows generated by BV vector fields developed in \cite{BV vector}, see Theorem 3.9 in section 3.3. In the case when $\nu$ is singular and may have unbounded support, we generalize the notion of weak Lagrangian solutions and prove their existence  with suitable initial data, see Theorem 4.6. As a byproduct, we obtain the existence of measure-valued dual space solutions when the initial dual density $\nu_0\in\mathcal{P}_2(\mathbf{R}^3)$  with support contained in $\mathbf{R}^2\times[-\frac{1}{\delta},-\delta]$ for some $\delta>0$, see Corollary 4.13.

\section{The study of the functional $E_{\nu}(h,\gamma)$}
\subsection{The case when $\nu$ has bounded support}

In this section, we study the functional involved in the geostrophic energy and the associated dual problem, prove basic properties such as unique existence of optimizers. Finally we give an alternative proof of dual space existence result using time stepping since later on we  will need some regularity properties of dual space solutions which are not so clear in Hamiltonian ODE approach as was done in \cite{Pelloni}.

We study the property of the functional
\begin{equation}
E_{\nu}(h,\gamma)=\int_{\Omega_{\infty}\times\Lambda}c(x,y)d\gamma(x,y)=\int_{\Omega_{\infty}\times\Lambda}[\frac{1}{2}(x_1^2+x_2^2+y_1^2+y_2^2)-x\cdot y]d\gamma
\end{equation}
where  $ (h,\gamma)\in\mathcal{M}_{\nu}$  $\nu\in\mathcal{P}_2(\mathbf{R}^3)$ and $supp\,\,\nu\subset\Lambda\subset\mathbf{R}^2\times[-\frac{1}{\delta},-\delta]$ and compact. Here let`s choose $\Lambda=B_D(0)\times[-\frac{1}{\delta},-\delta]$ and assume for technical reasons $\Omega_2\subset B_D(0)$\\
where $$\mathcal{M}_{\nu}=\{(h,\gamma)|h\geq0, \textrm{ continuous}, \int_{\Omega_2}hdx_1dx_2=1,\gamma\in
\Gamma(\sigma_h,\nu)\}$$
and also the functional
\begin{equation}
J_{\nu}^H(P,R)=\int_{\Lambda}[\frac{1}{2}(y_1^2+y_2^2)-R(y)]\nu(y)dy+\inf_{0\leq h\leq H}\int_{\Omega_{\infty}}[\frac{1}{2}(x_1^2+x_2^2)-P(x)]\sigma_h(x)dx
\end{equation}
where $$P(x)+R(y)\geq x\cdot y\,\,\,\forall x\in\Omega_2\times [0,H]\,\,\forall y\in\Lambda$$
Call the collection of all such pairs satisfying the above condition to be $\mathcal{N}_{\nu}$.

We will also consider the untruncated version, namely
\begin{equation}
J_{\nu}(P,R)=\int_{\Lambda}[\frac{1}{2}(y_1^2+y_2^2)-R(y)]\nu(y)dy+\inf_{h\geq0}\int_{\Omega_{\infty}}[\frac{1}{2}(x_1^2+x_2^2)-P(x)]\sigma_h(x)dx
\end{equation}
where we require
$$R\in L^1(d\nu)\,\,\,\,P\in L^1(\Omega_K)\,\,\forall K>0,\,\,\,and\,\,\,h\in L^1(\Omega_2)\,\,\,P\sigma_h
\in L^1(\Omega_{\infty})$$
and
$$P(x)+R(y)\geq x\cdot y\,\,\,\forall x\in\Omega_{\infty}\,\,y\in\Lambda$$
We will see later that $J_{\nu}(P,R)$ is dual to $E_{\nu}(h,\gamma)$ in the next subsection. Also the study of the dual problem $J_{\nu}(P,R)$ will help with proving uniqueness of minimizers of $E_{\nu}(h,\gamma)$, since the geostrophic functional($E_{\nu}(h,\gamma)$ in our case does not seem to have strict convexity as in \cite{Cullen-Gangbo},\cite{Cullen-Maroofi}. This was first noted in \cite{Pelloni} and the idea of considering a dual variational problem is inspired by \cite{axisymmtricflow}.

 In the case when $\nu$ has bounded support, $\Lambda$ can be taken to be bounded, and later on(in this subsection) we will show the $h$ which assumes the infimum in $(2.3)$ has a universal bound in $L^{\infty}$(depending only on the data of the problem) and so it will be equivalent to solving the truncated problem $J_{\nu}^H(P,R)$ if one takes $H$ large enough depending only on $\delta,\Omega_2$ and $\Lambda$.

Suppose $\partial_{x_3}P(x)\leq-\delta$, let`s define
$$\Pi_P(x_1,x_2,s):=\int_0^s[\frac{1}{2}(x_1^2+x_2^2)-P(x_1,x_2,x_3)]dx_3$$

It`s easy to see for fixed $(x_1,x_2)\in\Omega_2$, the function $s\longmapsto\Pi_p(x_1,x_2,s)$ is  uniformly convex and so there exists a unique
$s_*$ where $\Pi_P$ achieves minimum on $[0,\infty)$. We define this function to be $h_P(x_1,x_2)$. We also define $h_P^H(x_1,x_2)$ to be the unique $s^*\in[0,H]$ where $\Pi_P$ achieves minimum on $[0,H]$. Notice by convexity,one has $h_P^H=\min(h_P,H)$.

\begin{rem}
Whenever $h_P(x_1,x_2)>0$, we must have $$P(x_1,x_2,h_P(x_1,x_2))=\frac{1}{2}(x_1^2+x_2^2)$$
Otherwise
$$P(x_1,x_2,0)\leq\frac{1}{2}(x_1^2+x_2^2)$$
Conversely, if P is defined on $\Omega_{\infty}$ and $h$ satisfies above condition, then we will also have $h=h_P$

\end{rem}
\begin{rem}
It`s easy to see in the situation of $J_{\nu}(P,R)$
\begin{equation}
\inf_{h\geq0}\int_{\Omega_{\infty}}[\frac{1}{2}(x_1^2+x_2^2)-P(x)]\sigma_h(x)dx=\int_{\Omega_{\infty}}[\frac{1}{2}(x_1^2+x_2^2)-P(x)]\sigma_{h_P}(x)dx
\end{equation}
and in the situation of $J_{\nu}^H(P,R)$
\begin{equation}
\inf_{0\leq h\leq H}\int_{\Omega_H}[\frac{1}{2}(x_1^2+x_2^2)-P(x)]\sigma_h(x)dx=\int_{\Omega_H}[\frac{1}{2}(x_1^2+x_2^2)-P(x)]\sigma_{h_P^H}(x)dx
\end{equation}
\end{rem}
Now we prove the following
\begin{lem}
Suppose there exists a sequence $\partial_{x_3}P_n,\partial_{x_3}P\leq-\delta$, and $P_n\rightarrow P$ uniformly on $\Omega_2\times[0,H]$ for each $H>0$, then we have $h_{P_n}(x_1,x_2)\rightarrow h_P(x_1,x_2)$ uniformly on $\Omega_2$. If $P_n,P$ satisfy the same condition but is only defined on $\Omega_2\times[0,H_0]$, then $h_{P_n}^{H_0}\rightarrow h_{P}^{H_0}$ uniformly on $\Omega_2$
\end{lem}
\begin{proof}
Denote $h_n=h_{P_n}^{H_0}$, $h=h_{P}^{H_0}$ for convenience. First we show that $\{h_n\}$ is uniformly bounded. Indeed,
$$\frac{1}{2}(x_1^2+x_2^2)=P_n(x_1,x_2,h_{P_n}(x_1,x_2))\leq-\delta h_n(x_1,x_2)+P_n(x_1,x_2,0)$$

Put $\epsilon_n=\sup_{x\in\Omega_2\times[0, H]}|P_n(x)-P(x)|$, then we have

If $h_P(x_1,x_2)=0$, then
$$\frac{1}{2}(x_1^2+x_2^2)\geq P(x_1,x_2,0)\geq-\epsilon_n+(P_n(x_1,x_2,0)-P_n(x_1,x_2,h_n(x_1,x_2))+\frac{1}{2}(x_1^2+x_2^2)$$
$$\geq-\epsilon_n+\delta h_n(x_1,x_2)+\frac{1}{2}(x_1^2+x_2^2)$$

If $h_P(x_1,x_2)>0$,and $h_n\geq h$ then
$$0=P_n(x_1,x_2,h_n(x_1,x_2))-P(x_1,x_2,h(x_1,x_2))$$$$=P_n(x_1,x_2,h_n)-P(x_1,x_2,h_n)+P(x_1,x_2,h_n)-P(x_1,x_2,h)\leq\epsilon_n-\delta(h_n-h)$$

The other case $h_n\leq h$ can be dealt with similarly.
\end{proof}

We also need another lemma which gives control over the absolute bound for the maximizing sequence.
\begin{lem}
Suppose there exists constant $K>0$ and $P:\Omega_H\rightarrow \mathbf{R}$ with $P(x_2^*,x_2^*,0)=0$ where $(x_1^*,x_2^*)\in\Omega_2$ and $P,R$ are convex conjugate over the domain $\Omega_H$ and $\Lambda$ respectively, such that for some $\lambda>0$
$$-K\leq J_{\nu}^H(P-\lambda,R+\lambda)$$
and
$$|\nabla_2P|\leq K\,\,\,\,-\frac{1}{\delta}\leq\partial_{x_3}P\leq-\delta$$
then $$|\lambda|\leq C_1(K,\Lambda,\Omega_2,H)$$as long as $H\geq \frac{2}{\mathcal{L}^2(\Omega_2)}$

\end{lem}
\begin{proof}
First from our assumption on $P$, we obtain
$$|P(x)|\leq Kdiam\,\,\Omega_2+\frac{H}{\delta}:=C_{-1}\,\,\,x\in\Omega_H$$

Since $(P,R)$ are convex conjugate, one has
$$R(y)=\sup_{x\in\Omega_H}(x\cdot y-P(x))\geq x_1^*y_1+x_2^*y_2\geq -\max_{\Omega_2}|x|\cdot\max_{y\in\Lambda}(|y_1|+|y_2|):=-C_0$$

From the definition of $ J^H$, by taking $h=0$, one has
$$-K\leq\int_{\Lambda}[\frac{1}{2}(y_1^2+y_2^2)-R(y)-\lambda]\nu(y)dy\leq \max_{y\in\Lambda}(|y_1|+|y_2|)^2+C_0-\lambda$$

So$$\lambda\leq C_0+K+\max_{y\in\Lambda}(|y_1|+|y_2|)^2$$

On the other hand  we take $h=\frac{2}{\mathcal{L}^2(\Omega_2)}$, we then have
$$-K\leq\int_{\Lambda}[\frac{1}{2}(y_1^2+y_2^2)-R(y)-\lambda]\nu(y)dy+\int_{\Omega_2\times[0,\frac{2}{\Omega_2}]}[\frac{1}{2}(x_1^2+x_2^2)-P(x)+\lambda]dx$$$$\leq \max_{y\in\Lambda}(|y_1|+|y_2|)^2+C_{-1}+C_0+\max_{\Omega_2}|x|^2+\lambda$$
\end{proof}
Now we prove the existence of a pair of maximizer of $J^H_{\nu}$, using a standard compactness argument.
\begin{thm}
Suppose $H\geq \frac{2}{\mathcal{L}^2(\Omega_2)}$, then the variational problem $J^H_{\nu}$ has a  maximizer $(P,R)$, where $P$ and $R$ are convex conjugate to each other over $\Omega_H$ and $\Lambda$
\end{thm}
\begin{proof}
We choose a maximizing sequence
$(P_n,R_n)$, without loss of generality, we can assume they are convex conjugate over the domain $\Omega_H$ and $\Lambda$, then their derivatives are uniformly bounded(with a bound depending on $\Omega_2,\Lambda,H$) and P satisfies $\frac{-1}{\delta}\leq\partial_{x_3}P\leq-\delta$.

Now the functions $$\hat{P}_n:=P_n-P_n(x_1^*,x_2^*,0)\,\,\,\,\hat{R}_n:=R_n+P_n(x_1^*,x_2^*,0)$$
satisfy the assumption of previous lemma with $\lambda=-P_n(x_1^*,x_2^*,0)$, we can then conclude
$$|P_n(x_1^*,x_2^*,0)|\leq C$$

Since the derivatives are bounded independent of $n$, we get $P_n$ is uniformly bounded independent of $n$. Also since $P_n,R_n$ are convex conjugate, we see $R_n$ are also bounded uniformly independent of $n$.

Now we can apply Arzela-Ascoli to get a pair $(P,R)$ which is also convex conjugate over $\Omega_H$ and $\Lambda$.

Also by lemma 2.8, we also have

$$h_{P_n}^H\rightarrow h_P^H$$ uniformly on $\Omega_2$. Since we also have
$$P_n\rightarrow P\,\,\,\,R_n\rightarrow R\,\,\,\textrm{uniformly}$$
One can see $$J_{\nu}^H(P_n,R_n)\rightarrow J_{\nu}^H(P,R)$$

Hence $(P,R)$ is a maximizer.
\end{proof}

\begin{lem}
Suppose for some $H>0$, the variational problem $J^H_{\nu}$ has a pair of  maximizer $(P,R)$, such that $P(x)=\sup_{y\in\Lambda}(x\cdot y-R(y))$ and put $h=h_P^H$, then $$\nabla P_{\sharp}\sigma_h=\nu$$ in particular,
$$\int_{\Omega_2}h=1$$
\end{lem}
\begin{proof}
Let $g(y)\in C_b(\Lambda)$,$\delta>0$, define
$$R_{\delta}(y):=R(y)+\delta g(y)\,\,\,\,\,\,\,P_{\delta}(x):=\sup_{y\in\Lambda}(x\cdot y-R_{\delta}(y))$$

Since $(P,R)$ is a pair of maximizers, we have
$$J^H(P_{\delta},R_{\delta})\leq J^H(P,R)$$i.e, we have
$$-\delta\int_{\Lambda}g(y)d\nu(y)\leq\inf_{0\leq h\leq H}\int_{\Omega_h}[\frac{1}{2}(x_1^2+x_2^2)-P(x)]dx-\inf_{0\leq h\leq H}\int_{\Omega_h}[\frac{1}{2}(x_1^2+x_2^2)-P_{\delta}(x)]dx$$
$$\leq\int_{\Omega_2}dx_1dx_2\int_0^{h_{\delta}}[P_{\delta}(x_1,x_2,x_3)-P(x_1,x_2,x_3)]dx_3$$
Here $$h_{\delta}(x_1,x_2)=h^H_{P_{\delta}}(x_1,x_2)$$

Note that $\Pi_{P_{\delta}}(x_1,x_2,s)$ achieves minimum over $[0,H]$ at $h_{\delta}$. Therefore $h_{\delta}$ achieves the second infimum above.

Then we notice that since $R_{\delta}\rightarrow R$ uniformly on $\Lambda$, we have $P_{\delta}\rightarrow P$ uniformly on $\Omega_H$. Hence by Lemma 2.8, we have
$$h_{\delta}(x_1,x_2)\rightarrow h(x_1,x_2) \,\,\,\,\,\,\textrm{uniformly}$$

Suppose $P $ is differentiable at $x$, and let $y_{\delta}\in\bar{\Lambda}$ be the point such that
$$P_{\delta}(x)=x\cdot y_{\delta}-R_{\delta}(y_{\delta})$$

then we have
$$y_{\delta}\rightarrow\nabla P(x)\textrm{ as $\delta\rightarrow0$}$$

Also notice that
$$-g(\nabla P(x))\leq\frac{P_{\delta}(x)-P(x)}{\delta}\leq-g(y_{\delta})$$

By letting $\delta\rightarrow0$,we obtain
$$-\int_{\Lambda}g(y)\nu(y)dy\leq-\int_{\Omega_h}g(\nabla P(x))dx$$

Replacing $g$ by $-g$, we are done.
\end{proof}
Now we prove the function $h_P^H$ obtained above is Lipschitz.
\begin{prop}
Let $(P,R)$ be a pair of convex conjugate maximizers over the domain $\Omega_2\times[0,H]$ and $\Lambda$ respectively, then $h_P^H(x_1,x_2)$ is Lipschitz, with a Lipschitz constant depending only on $\delta,\Lambda,\Omega_2$(not $H$)
\end{prop}
\begin{proof}
Let`s denote $h_P^H(x_1,x_2)=h$. The proof is based on the following facts.

$(i)$$P(x_1,x_2,h(x_1,x_2))=\frac{1}{2}(x_1^2+x_2^2)$ whenever $0<h<H$

$(ii)$$\partial_{x_3}P\leq-\delta$, and $|\nabla_2P|\leq C$. Here $C$ depend only on $\Lambda$

Pick $(x_1,x_2),(z_1,z_2)\in\Omega_2$, without loss of generality, we can assume $h(z_1,z_2)<h(x_1,x_2)$

If $0<h(z_1,z_2)<h(x_1,x_2)<H$, then we have
$$P(x_1,x_2,h(x_1,x_2))=\frac{1}{2}(x_1^2+x_2^2)\,\,\,P(z_1,z_2,h(z_1,z_2))=\frac{1}{2}(z_1^2+z_2^2)$$

Thus
$$-C(|x_1-z_1|+|x_2-z_2|)\leq P(x_1,x_2,h(x_1,x_2))-P(x_1,x_2,h(z_1,z_2))+P(x_1,x_2,h(z_1,z_2))+P(z_1,z_2,h(z_1,z_2))$$
$$\leq-\delta(h(x_1,x_2)-h(z_1,z_2))+C(|x_1-z_1|+|x_2-z_2|)$$

Next consider $0<h(z_1,z_2)<h(x_1,x_2)=H$, then we have
$$P(x_1,x_2,H)\geq\frac{1}{2}(x_1^2+x_2^2)\,\,\,\,P(z_1,z_2,h(z_1,z_2))=\frac{1}{2}(z_1^2+z_2^2)$$

Thus
$$-C(|x_1-z_1|+|x_2-z_2|)\leq P(x_1,x_2,H)-P(z_1,z_2,h(z_1,z_2))$$
$$\leq C(|x_1-z_1|+|x_2-z_2|)-\delta(H-h(z_1,z_2))$$

Notice that if $h(z_1,z_2)=0$, then $P(z_1,z_2,0)\leq\frac{1}{2}(z_1^2+z_2^2)$, see Remark 2.4, so other cases can be dealt with in a similar way.
\end{proof}
\begin{cor}
Suppose $H\geq\frac{2}{\mathcal{L}^2(\Omega_2)}$(such that $J_{\nu}^H(P,R)$ has maximizers by Theorem 2.10), and let $(P,R)$ be convex conjugate maximizers defined on $\Omega_H$ and $\Lambda$ respectively, then there exists a constant $C=C(diam\,\,\Omega_2,\delta,\Lambda)$, such that
$$h_P^H(x_1,x_2)\leq C$$
\end{cor}
\begin{proof}
By Lemma 2.11, we have $$\int_{\Omega_2}h=1$$ so there exists $(z_1,z_2)\in\Omega_2$, such that $h(z_1,z_2)\leq\frac{2}{\mathcal{L}^2(\Omega_2)}$.

By above corollary, we know $h$ is Lipschitz, and since $$P(x_1,x_2,h)=\frac{1}{2}(x_1^2+x_2^2)\,\,\,\textrm{whenever   $h>0$}$$so $$\nabla h(x_1,x_2)=\frac{1}{\partial_{x_3}P}(x_1-\partial_1P,x_2-\partial_2P)$$  Thus
$$|\nabla h|\leq\frac{\max_{\Omega_2}|x|+\max_{\Lambda}(|y_1|+|y_2|)}{\delta}$$
So
\begin{equation}
|h|\leq\frac{2}{\mathcal{L}^2(\Omega_2)}+\frac{2\max_{\Omega_2}|x|+\max_{\Lambda}(|y_1|+|y_2|)}{\delta}\cdot diam\,\,\Omega_2
\end{equation}
\end{proof}
As an easy consequence, we deduce:
\begin{cor}
Suppose $H>C_0$, here the $C_0$ is the constant on the right hand side of (2.14), if $(P,R)$ is a pair of maximizer of $J_{\nu}^H(P,R)$, convex conjugate over $\Omega_H$ and $\Lambda$, then $0\leq h_P^H<H$.
\end{cor}

\begin{cor}
Suppose $H$ satisfy the same condition as in previous corollary, then there exists another constant $C=C(\delta,\Omega_2,\Lambda,H)$, such that if $(P,R)$ is a convex conjugate maximizer of $J^H_{\nu}(P,R)$, we have $|P(x)|\leq C,\forall x\in\Omega_H$
\end{cor}
\begin{proof}
Fix $(x_1,x_2)\in\Omega_2$, if $h(x_1,x_2)>0$, we then have $$\frac{1}{2}(x_1^2+x_2^2)\leq P(x_1,x_2,0)$$

On the other hand
$$\frac{1}{2}(x_1^2+x_2^2)=P(x_1,x_2,h(x_1,x_2))
=P(x_1,x_2,0)+\int_0^h\partial_{x_3}P(x_1,x_2,s)ds$$$$\geq-\frac{1}{\delta}h(x_1,x_2)+P(x_1,x_2,0)$$

Thus by Corollary 2.13, we know $$|P(x_1,x_2,0)|\leq C\,\,\,such\,\, that\,\,h(x_1,x_2)>0 $$ where C has the said dependence.

Now notice that since by assumption $(P,R)$ are convex conjugate over the domain
$\Omega_H$ and $\Lambda$ respectively, we know
$$\partial P(\Omega_H)\subset \Lambda$$
Hence $$|\nabla P(x)|\leq C(\Lambda,\delta)\,\,a.e$$

Fix $(z_1,z_2)$ such that $h(z_1,z_2)>0$, then for $x\in\Omega_H$, we have
$$P(x_1,x_2,x_3)=P(z_1,z_2,0)+\int_0^1\nabla P(t(z_1,z_2,0)+(1-t)(x_1,x_2,x_3))\cdot(x_1-z_1,x_2-z_2,x_3)dt$$

The result follows easily.
\end{proof}
In the sequel, we will always assume
\begin{equation}
H>\frac{2}{\mathcal{L}^2(\Omega_2)}+\frac{2\max_{\Omega_2}|x|+\max_{\Lambda}(|y_1|+|y_2|)}{\delta}\cdot diam\,\,\Omega_2
\end{equation}
unless otherwise stated.

\begin{thm}
Suppose $H$ is as in (2.17), $\nu\in\mathcal{P}(\mathbf{R}^3)$, with $supp\,\,\nu\subset\Lambda$ then we have

$(i)$$E_{\nu}(h,\gamma)\geq J_{\nu}^H(P,R)$, for any $(h,\gamma)\in\mathcal{M}_{\nu}$ and $(P,R)$ satisfying $P(x)+R(y)\geq x\cdot y\,\,\,\forall x\in\Omega_H\,\,\,\forall y\in\Lambda$

$(ii)$Suppose $(P,R)$ is a convex conjugate maximizer of $J^H(P,R)$, then we have equality in $(i)$ if and only if $\gamma=(id\times\nabla P)_{\sharp}\sigma_h$ and $h=h_P^H$

$(iii)$In the situation of  $(ii)$, and if $\nu<<\mathcal{L}^3$, we also have $$\nabla R_{\sharp}\nu=\sigma_h$$
\end{thm}
\begin{proof}
First we prove $(i)$.

Without loss of generality, we can assume $(P,R)$ be a pair of convex conjugate maximizers of $J^H_{\nu}(P,R)$, and we can naturally extend $P$ to be defined on $\Omega_{\infty}$ such that $(P,R)$ are again convex conjugate. Indeed,
$$P(x)=\sup_{y\in\Lambda}(x\cdot y-R(y))$$
and we define $P(x)$ for $x\in\Omega_{\infty}$ by the same formula. Then one can check
$$R(y)=\sup_{x\in\Omega_H}(x\cdot y-P(x))=\sup_{x\in\Omega_{\infty}}(x\cdot y-P(x))$$

This implies in particular
$$P(x)+R(y)\geq x\cdot y\,\,\,x\in\Omega_{\infty}\,\,\,y\in\Lambda$$
Therefore, assuming $(h,\gamma)\in\mathcal{M}$, we can write
$$\int c(x,y)d\gamma(x,y)\geq\int_{\Lambda}[\frac{1}{2}(y_1^2+y_2^2)-R(y)]\nu(y)dy+\int_{\Omega_{\infty}}[\frac{1}{2}(x_1^2+x_2^2)-P(x)]\sigma_h(x)dx$$$$\geq\int_{\Lambda}[\frac{1}{2}(y_1^2+y_2^2)-R(y)]\nu(y)dy+\int_{\Omega_{\infty}}[\frac{1}{2}(x_1^2+x_2^2)-P(x)]\sigma_{h_P^H}(x)dx=J^H(P,R)\,\,\,(*)$$

Here only the second inequality requires some explanation.

By Corollary 2.15 and by strict convexity of $\Pi_P(x_1,x_2,s)$ in s, we see that $s\longmapsto\Pi_P(x_1,x_2,s)$ attains minimum over $[0,\infty)$ at $h_P^H(x_1,x_2)$. So we have
$$\int_{\Omega_{\infty}}[\frac{1}{2}(x_1^2+x_2^2)-P(x)]\sigma_{h}(x)dx=\int_{\Omega_2}\Pi_P(x_1,x_2,h(x_1,x_2))dx_1dx_2$$$$\geq
\int_{\Omega_2}\Pi_P(x_1,x_2,h_P^H(x_1,x_2))dx_1dx_2=\int_{\Omega_{\infty}}[\frac{1}{2}(x_1^2+x_2^2)-P(x)]\sigma_{h_P^H}(x)dx$$

It`s easy to see above inequality takes equality if and only if $h(x_1,x_2)=h_P^H(x_1,x_2)$.
 Up to now,we proved $(i)$

Now we can prove $(ii)$. Suppose we have $E_{\nu}(h,\gamma)=J_{\nu}^H(P,R)$, then both inequality in $(*)$ must be equality. Thus we know $h=h_P^H$ and we know from Corollary 2.13 that $supp\,\,\sigma_h\subset\Omega_H$. The first inequality in $(*)$ takes equality if and only if $P(x)+R(y)=x\cdot y\,\,\,\gamma-a.e\,\,(x,y)$, or equivalently $y\in\partial P(x)\,\,\gamma-a.e\,\,(x,y)$. So we obtain $\gamma=(id\times\nabla P)_{\sharp}\sigma_h$. Conversely, if $h=h_P^H$ and $\gamma=(id\times\nabla P)_{\sharp}\sigma_h$, then by Lemma 2.11, $(h,\gamma)\in\mathcal{M}_{\nu}$, also we have $E_{\nu}(h,\gamma)=J^H_{\nu}(P,R)$.

To see $(iii)$, note that by $(ii)$, we have $$\nabla P_{\sharp}\sigma_h=\nu$$

And since $\nu<<\mathcal{L}^3$, we know $$\nabla R\circ\nabla P(x)=x\,\,\,\,\sigma_h\,\,a.e$$

Indeed put $$E=\{y\in\Lambda|\nabla R\textrm{ is not defined at $y$}\}$$
then
$$\nu(E)=\sigma_h((\nabla P)^{-1}(E))=0$$
And if $\nabla P$ is defined at $x$, $\nabla R$ is defined at $\nabla P(x)$, then one has $\nabla R(\nabla P(x))=x$.

So
$$\sigma_h=\nabla R_{\sharp}\nabla P_{\sharp}\sigma_h=\nabla R_{\sharp}\nu$$
\end{proof}
Now we can prove the unique existence of the variational problem $E_{\nu}(h,\gamma)$.
\begin{cor}
Suppose H is as (2.17), then

$(i)$ $E_{\nu}(h,\gamma)$ has a unique minimizer $(h,\gamma)$;

$(ii)$Suppose $(P_0,P_1)$,$(P_1,R_1)$ are maximizers of $J_{\nu}^H(P,R)$, convex conjugate over $\Omega_H$ and $\Lambda$, then $h_{P_0}^H=h_{P_1}^H:=h$, and $P_0=P_1$ on $\Omega_{h}$

\end{cor}
\begin{proof}
Existence of at least one maximizer of $J_{\nu}^H(P,R)$ has been proved in Theorem 2.10. Theorem 2.18 $(ii)$ gives the existence of at least one minimizer of $E_{\nu}(h,\gamma)$. Now we show uniqueness.

First we fix a maximizer of $J_{\nu}^H(P,R)$, say $(P_0,R_0)$, which is convex conjugate over $\Omega_H$ and $\Lambda$. If $(h_0,\gamma_0)$,$(h_1,\gamma_1)$ are both minimizers of $E_{\nu}(h,\gamma)$, then we have by previous theorem
$$h_0=h_1=h_{P_0}^H(x_1,x_2):=h_0\,\,\,\,,and\,\,\,\, \gamma_1=\gamma_0=(id\times\nabla P_0)_{\sharp}\sigma_{h}$$

This proves the uniqueness of the minimizer of $E_{\nu}(h,\gamma)$.

To see the uniqueness of $J^H(P,R)$ on $\Omega_{h_0}$, say both $(P_0,R_0)$,$(P_1,R_1)$ are maximizers. Let $h_0=h_{P_0}^H$ and $\gamma_0=(id\times\nabla P_0)_{\sharp}\sigma_{h_0}$. Then we know by Proposition 2.12 that $h_0$ is Lipschitz, also we know by Theorem 2.18$(ii)$ that $h=h_{P_1}^H=h_{P_0}^H$

Also by uniqueness of minimizer for $E_{\nu}(h,\gamma)$ already proved above, we know from above theorem $(ii)$
$$(id\times\nabla P_0)_{\sharp}\sigma_{h_0}=(id\times\nabla P_1)_{\sharp}\sigma_{h_0}$$
This implies $$\nabla P_0=\nabla P_1\,\,\,\sigma_{h_0}\,\,a.e$$

Let $U\subset \Omega_2$ to be a connected component of the open set $\{(x_1,x_2)\in\Omega_2|h_0(x_1,x_2)>0\}$, then the set $U_h:=\{(x_1,x_2,x_3)|(x_1,x_2)\in U,0<x_3<h_0(x_1,x_2)\}$ is connected and so we have $$P_0=P_1+C_U\,\,\,on\,\,\,U_h$$

But we also have $$P_0(x_1,x_2,h_0(x_1,x_2))=P_1(x_1,x_2,h_0(x_1,x_2))=\frac{1}{2}(x_1^2+x_2^2)$$for $(x_1,x_2)\in U$, hence $C_U=0$.
So we have $$P_0=P_1\,\,\,\sigma_{h_0} a.e$$
\end{proof}
\begin{cor}
Let $\Lambda=B_D(0)\times[-\frac{1}{\delta},-\delta]$ and assume also that $\Omega_2\subset B_D(0)$, H is chosen as (2.17). Then there exists a unique maximizer $(P_2,R_2)$ of $J_{\nu}^H(P,R)$ with the following properties:

$(i)$$(P_2,R_2)$ are convex conjugate over both $\Omega_{h_0}\bigcup\{x_3=0\},\Lambda$ and $\Omega_H,\Lambda$.

$(ii)$$P_2(x_1,x_2,0)=\frac{1}{2}(x_1^2+x_2^2)$ whenever $h_0(x_1,x_2)=0$
\end{cor}
\begin{proof}
The uniqueness is easy to see, since by $(i)$, we must have
\begin{equation}
R_2(y)=\sup_{x\in\Omega_{h_0}\bigcup\{x_3=0\}}(x\cdot y-P_2(x))\,\,\,\,\textrm{$\forall y\in\Lambda$}
\end{equation}
The previous corollary shows that for any maximizer $(P,R)$ convex conjugate over $\Omega_H$ and $\Lambda$, the value of $P$ must agree on $\bar{\Omega}_h$, while point $(ii)$ takes care of the part where $x_3=0$ and $h(x_1,x_2)=0$. Hence $R_2$ is uniquely defined by (2.21).

And $$P_2(x)=\sup_{y\in\Lambda}(x\cdot y-R_2(y))\,\,\,\forall x\in\Omega_H$$
because $(P_2,R_2)$ are assumed to be convex conjugate over $\Omega_{H},\Lambda$, so $P_2$ is uniquely defined over $\Omega_H$.

The existence of such a maximizer is more technical and is proved in the appendix.
\end{proof}
To conclude this section, we notice the following stability results. Their proofs are standard compactness argument, making use of the uniqueness proved in Corollary 2.19 and 2.20. We sketch the proofs.
\begin{lem}
Let $\nu_n\rightarrow\nu$ narrowly, then we have
$$\inf_{(h,\gamma)\in\mathcal{M}_{\nu}}E_{\nu}(h,\gamma)\leq\lim\inf_{n\rightarrow\infty}\inf_{(h,\gamma)\in\mathcal{M}_{\nu_n}}E_{\nu_n}(h,\gamma)$$
\end{lem}
\begin{proof}
Notice the uniform Lipschitz continuity and boundedness of $h_n$ proved in Proposition 2.12 and Corollary 2.13, and weak compactness of $\gamma_n$, finally make use of the uniqueness given by Corollary 2.19 $(i)$.
\end{proof}
\begin{lem}
Let $\nu_n\rightarrow\nu$ narrowly, then we have
$$\sup_{(P,R)\in\mathcal{N}}J^H_{\nu}(P,R)\geq\lim\sup_{n\rightarrow\infty}\sup_{(P,R)\in\mathcal{N}_{\nu_n}} J_{\nu_n}^H(P,R)$$
\end{lem}
\begin{proof}
Choose maximizers $(P_n,R_n)$ according to Corollary 2.20, using the boundedness given by Corollary 2.16 and that they are uniformly Lipschitz continuous.
\end{proof}
\begin{cor}
Suppose $\nu_n\rightarrow\nu$ narrowly, then $$\lim_{n\rightarrow\infty}\inf_{(h,\gamma)\in\mathcal{M}_{\nu_n}}E_{\nu_n}(h,\gamma)=\inf_{(h,\gamma)\in\mathcal{M}_{\nu}}E_{\nu}(h,\gamma)$$
\end{cor}
\begin{proof}
$"\geq"$ follows from Lemma 2.23.$"\leq"$ follows from Lemma 2.22. Recall that by Theorem 2.18$(ii)$, we have $$\inf_{(h,\gamma)\in\mathcal{M}_{\nu}}E_{\nu}(h,\gamma)=\sup_{(P,R)\in\mathcal{N}_{\nu}}J_{\nu}^H(P,R)$$
\end{proof}
 Next we prove a stability result of the optimizers under narrow convergence. It will be useful when one proves the continuity in time of certain quantities.
\begin{thm}
Suppose $\nu_n,\nu\in\mathcal{P}(\Lambda)$ and $\nu_n\rightarrow\nu$ narrowly. Let $(P_n,R_n),(P,R)$ be the unique maximizers of $J^H_{\nu_n},J^H_{\nu}$ given by Corollary 2.20, $(h_n,\gamma_n)$,$(h_0,\gamma_0)$ be the minimizers of $E_{\nu_n},E_{\nu}$ respectively, then the following holds.

$(i)$ $h_n\rightarrow h_0$ uniformly,

$(ii)$$P_n\rightarrow P$ in $W^{1,r}(\Omega_H)$, for any $r<\infty$

$(iii)$$R_n\rightarrow R$ in $W^{1,r}(\Lambda)$, for any $r<\infty$
\end{thm}
\begin{proof}
This is proved in a similar way as Lemma 2.22 and Lemma 2.23. We simply note that by convexity, the uniform convergence of $P_n$,$R_n$ implies the a.e convergence of $\nabla P_n$,$\nabla R_n$, and they are uniformly bounded, hence the claimed convergence follows.
\end{proof}
Next we apply the above obtained results to give an alternative proof for the existence of weak solutions in dual spaces with absolute continuous initial data, which was first done in \cite{Pelloni}, using Hamiltonian ODE approach. Our approach here is less involved.

Now we define
\begin{equation}
\Lambda_0=B_{D_0}(0)\times[-\frac{1}{\delta},-\delta]\,\,\,and\,\,\,\Lambda=B_D(0)\times[-\frac{1}{\delta},-\delta]
\end{equation}

Based on what we proved above about the properties of the functional $E_{\nu}(h,\gamma)$ and $J_{\nu}^H(P,R)$ in Theorem 2.18, the dual space problem can be reformulated as
\begin{equation}
\partial_t\nu+\nabla\cdot(\nu\mathbf{w})=0
\end{equation}
\begin{equation}
\mathbf{w}=J(y-\nabla P^*(t,y))\,\,\,y\in\mathbf{R}^3
\end{equation}
\begin{equation}
(h,(id\times\nabla P)_{\sharp}\sigma_h)\,\,\,minimizes\,\,\,E_{\nu}(h,\gamma)\,\,\,among\,\,\,\mathcal{M}_{\nu(t,\cdot)}
\end{equation}

In the above, $P^*(t,y)=\sup_{x\in\Omega_H}(x\cdot y-P(x))$ for some large enough H\\
\begin{thm}
Let $1<q<\infty$, $T>0$ be given and $\nu_0\in L^q(\Lambda_0)$, where $\Lambda_0,\Lambda$ defined as above, suppose also \
\begin{equation}
D>D_0+\max_{\Omega_2}|x|\cdot (T+1)
\end{equation}
and
\begin{equation}
H>\frac{2}{\mathcal{L}^2(\Omega_2)}+\frac{2\max_{\Omega_2}|x|+2D}{\delta}\cdot diam\,\,\Omega_2
\end{equation}
Then there exists a weak solution to (2.27)-(2.29), $(h,P,R)$ with $\nu(t,\cdot):=\nabla P_{t\sharp}\sigma_{h(t,\cdot)}$, $R=P^*$ such that

$$(i)\textrm{$t\longmapsto \nu(t,\cdot)\in\mathcal{P}_{ac}(\mathbf{R}^3)$ is narrowly continuous with supp $\nu\subset\Lambda$}$$
$$\textrm{$\nu(t,\cdot)\in L^q(\Lambda)$, and $||\nu(t,\cdot)||_{L^q(\Lambda)}= ||\nu_0(\cdot)||_{L^q(\Lambda)}$, $\forall t\in[0,T]$}$$
$$(ii)\textrm{$(P(t,\cdot),R(t,\cdot))$ is the unique maximizer of $J_{\nu(t,\cdot)}^H(P,R)$}$$$$
 \textrm {with  properties (i),(ii) in Corollary 2.20}$$
$$(iii)W_1(\nu(t_1,\cdot),\nu(t_2,\cdot))\leq C|t_1-t_2|$$
$$(iv)P(t,\cdot)\in L^{\infty}([0,T],W^{1,\infty}(\Omega_{\infty}))\bigcap C([0,T], W^{1,r}(\Omega_{\infty}))$$$$
R\in L^{\infty}([0,T],W^{1,\infty}(\Lambda))\bigcap C([0,T],W^{1,r}(\Lambda))\,\,\forall r<\infty$$
$$(v)Let\,\, \gamma(t,\cdot)=(id\times\nabla P)_{\sharp}\sigma_h,\,\,(h,\gamma)\,\,are\,\, minimizers \,\,of\,\, E_{\nu}(h,\gamma) \,\,and$$$$\,\, h\in L^{\infty}([0,T],W^{1,\infty}(\Omega_2))\bigcap C([0,T]\times\bar{\Omega}_2)$$

\end{thm}
\begin{proof}
Let $j_s:\mathbf{R}^3\rightarrow\mathbf{R}$ be standard mollifier defined by
$$j_s(y)=\frac{1}{s^3}j(\frac{y}{s})$$

First we mollify the initial data by defining
$$\nu_s^0(y)=j_s*\nu_0(y)$$

Then for $0<s<\frac{\delta}{2}$, we have
$$supp\,\,\nu_s^0\subset B_{D_0+s}(0)\times(-\frac{2}{\delta},-\frac{\delta}{2})$$

Next we  construct approximate solutions, here we need to control the speed of propagation of the support of approximate solutions. Define $$D^k_s=D_0+ks\max_{\Omega_2}|x|\,\,\,0\leq k\leq \frac{T}{s}$$

Given $\nu_s^k:=\nu_s(ks,\cdot)$ with $supp\,\,\nu_s^k\subset B_{D_s^k}(0)\times[-\frac{2}{\delta},-\frac{\delta}{2}]$, we obtain $\nu_s^{k+1}$ such that $supp\,\,\nu_s^{k+1}\subset B_{D_s^{k+1}}(0)\times[-\frac{1}{\delta},-\delta]$ in the following way.

Let $(h_s^k,\gamma_s^k)$ be the minimizer of $E_{\nu_s^k}(h,\gamma)$, Let $(P_s^k,R_s^k)$ the unique  maximizer of $ J_{\nu_s^k}^H(P,R)$ given by Corollary 2.20. This is possible by our assumption on $H$ and D, see (2.17),(2.31) and (2.32).\\
We set $\gamma_s^k=(id\times\nabla P_s^k)_{\sharp}\sigma_{h_s^k}$.

Define
$$\mathbf{w}_s^k(y)=J(y-\nabla R_s^k(y))$$
$$\mathbf{u}_s^k(y)=J(y-\nabla(j_s*R_s^k)(y))$$

Observe that $\mathbf{u}_s^k(y)$ is $C^{\infty}$ and divergence free and we obtain $\nu_s^{k+1}$ by solving the transport equation
$$\partial_t\nu_s=-\nabla\cdot(\mathbf{u}_s^k\nu_s^k)\,\,\,in\,\,\,[ks,(k+1)s)\times\mathbf{R}^3$$
$$\nu_s(ks,y)=\nu_s^k(y)\,\,\,in\,\,\,\mathbf{R}^3$$
and set $\nu_s^{k+1}=\nu_s((k+1)s,y)$.

In more detail, we solve the ODE
$$\frac{d\Phi_s^k(t,y)}{dt}=\mathbf{u}_s^k(\Phi_s^k(t,y))\,\,\,t\in[ks,(k+1)s],y\in\mathbf{R}^3$$
$$\Phi_s^k(ks,y)\equiv y$$
Then $$\nu_s(t,y)=\nu_s^k((\Phi_s^k)^{-1}(t,y))\,\,\,t\in[ks,(k+1)s]$$
since $\mathbf{u}_s^k$ is divergence free.

Therefore
$$supp\,\,\nu_s(t,\cdot)\subset \Phi_s^k(t,\cdot)(supp\,\,\nu_s^k)$$
Now take $y_0\in supp\,\,\nu_s^k$, we know
$$\frac{d|\Phi_s^k(t,y_0)|}{dt}=\frac{\Phi_s^k\cdot\frac{d\Phi_s^k}{dt}}{|\Phi_s^k|}=\frac{\Phi_s^k\cdot J(\Phi_s^k-\nabla(j_s*R_s^k)(\Phi_s^k))}{|\Phi_s^k|}$$
Now notice that
$$\Phi_s^k\cdot J\Phi_s^k=0\,\,\,\,\,[J(\nabla(j_s*R_s^k))]_3=0$$
Hence
$$|\Phi_s^k\cdot J(\Phi_s^k-\nabla(j_s*R_s^k))|\leq |\Phi_s^k||\nabla_2 R_s^k|\leq |\Phi_s^k|\max_{x\in\Omega_2}|x|$$
It follows that
$$\frac{d|\Phi_s^k(t,y_0)|}{dt}\leq\max_{\Omega_2}|x|\,\,\,\,t\in[ks,(k+1)s]$$
hence $$|\Phi_s^k((k+1)s,y_0)|\leq D_s^k+s\cdot\max_{\Omega_2}|x|\leq D_s^{k+1}$$
So we obtained $$supp\,\,\nu_s^{k+1}\subset B^2_{D_s^{k+1}}\times[-\frac{1}{\delta},-\delta]$$

Also we can define $P_s^{k+1},R_s^{k+1},h_s^{k+1},\gamma_s^{k+1},\mathbf{w}_s^{k+1},\mathbf{u}_s^{k+1}$ in the same way as above. Denote $\bar{P}_s(t,\cdot),\bar{R}_s(t,\cdot),\bar{h}_s(t,\cdot),\bar{\gamma}_s(t,\cdot)$ to be above defined piecewise constant function  in time, i.e, $\bar{P}_s(t,\cdot)\equiv P_s^k\,\,t\in[ks,(k+1)s)$, similar for others.

Recall the definition of D, we have shown that  if $s<\frac{\delta}{2}$
$$supp\,\,\nu_s(t,\cdot)\subset B_{D+s}(0)\times[-\frac{2}{\delta},-\frac{\delta}{2}]$$
Similar to the proof in \cite{Cullen-Gangbo}, one has
 $$W_1(\nu_s(t_1,\cdot),\nu_s(t_2,\cdot))\leq (D+\max_{x\in\Omega_2}|x|)||\nu_0||_{L^1(\mathbf{R}^3)}|t_1-t_2|$$
 where $W_1(\cdot,\cdot)$ is the 1-Wasserstein distance.
Follow the proof of \cite{Cullen-Maroofi} Theorem 5.3, we conclude that up to a subsequence
$$\nu_{s_j}\rightarrow\nu\,\,\,weakly\,\,\,in\,\,\, L^r([0,T]\times\mathbf{R}^3)$$
and $$\nu_{s_j}(t,\cdot)\rightarrow\nu(t,\cdot)\,\,\,weakly \,\,\,in\,\,\,L^r(\Lambda)$$
Also
$$\nu(t,\cdot)\in L^q([0,T]\times \Lambda)\,\,\,||\nu(t,\cdot)||_{L^q(\Lambda)}=||\nu_0||_{L^q(\Lambda)}$$
$$W_1(\nu(t_1,\cdot),\nu(t_2,\cdot))\leq(D+\max_{x\in\Omega_2}|x|)||\nu_0||_{L^1(\mathbf{R}^3)}|t_1-t_2|$$
Since $\nu_{s_j}(t,\cdot)\rightarrow \nu(t,\cdot)$ narrowly as measures, we conclude by Theorem 2.25 that $\nabla\bar{R}_{s_j}(t,\cdot)\rightarrow \nabla R(t,\cdot)$ in $L^r(\Lambda)$ and hence $\bar{\mathbf{u}}_{s_j}\rightarrow\mathbf{w}$ in $L^r(\Lambda)$.
$$\bar{\mathbf{u}}_{s_j}\nu_{s_j}(t,\cdot)\rightarrow\mathbf{w}\nu(t,\cdot)\,\,\,weakly\,\,\,in\,\,\,L^r(\Lambda;\mathbf{R}^3)$$
so we conclude
 and $\nu$ satisfies the equation $$\partial_t\nu+\nabla\cdot(\mathbf{w}\nu)=0\,\,\,;\nu(0,\cdot)=\nu_0$$ in the sense of distribution, where $\mathbf{w}=J(y-\nabla R(t,y))$ and $(P(t,\cdot),R(t,\cdot))$  is the minimizer of $J_{\nu(t,\cdot)}^H(P,R)$.
The property of $P,h$ follows from the stability result proved above and the narrow continuity of $\nu_t$
\end{proof}
\subsection{Generalization to $\nu$ with unbounded support}
In this subsection, we will consider the case when $\nu$ may have unbounded support and generalize the properties obtained in previous subsection. The result of this section will be used only in section 4, when the initial data has only $L^2$, instead of $L^{\infty}$ gradient. The ideas are quite similar, but certain complications arise.

We will always take in this subsection $\Lambda=\mathbf{R}^2\times[-\frac{1}{\delta},-\delta]$, and we assume $\nu\in\mathcal{P}_2(\mathbf{R}^3)$ with $supp\,\,\nu\subset\Lambda$. In this setting, $E_{\nu}(h,\gamma)$ is defined the same way as (2.1), $J_{\nu}(P,R)$ is defined as in (2.2). Obvious examples as $P(x)=\frac{1}{2}(x_1^2+x_2^2),R(y)=\frac{1}{2}(y_1^2+y_2^2)$ shows \begin{equation}
\sup_{(P,R)\in\mathcal{N}_{\nu}}J_{\nu}(P,R)\geq0
\end{equation}

Recall $\mathcal{N}_{\nu}$ is the set of pairs $(P,R)$ such that $P(x)+R(y)\geq x\cdot y$ with suitable integrability condition. See subsection 2.1.

Suppose $(P,R)\in\mathcal{N}_{\nu}$, we can then use double convexification to define
$$R_0(y)=\sup_{x\in\Omega_{\infty}}(x\cdot y-P(x))$$
and
$$P_0(x)=\sup_{y\in\Lambda}(x\cdot y-R_0(y))$$
we have
$$P_0\leq P\,\,\,and\,\,\,R_0\leq R$$
Therefore
$$J_{\nu}(P,R)\leq J_{\nu}(P_0,R_0)$$
then set $\hat{P}_0(x)=\max(P_0(x),\frac{1}{2}(x_1^2+x_2^2))$ and set
$$R_1(y)=\sup_{x\in\Omega_{\infty}}(x\cdot y-\hat{P}_0(x))$$
$$P_1(x)=\sup_{y\in\Lambda}(x\cdot y-R_1(y))$$
Since $\hat{P}_0(x)\geq\frac{1}{2}(x_1^2+x_2^2)$, it`s easy to see $R_1(y)\leq\frac{1}{2}(y_1^2+y_2^2)$. Since $\hat{P}_0(x)+R_1(y)\geq x\cdot y$, we see $P_1(x)\leq \hat{P}_0(x)$ and the definition of $P_1$ shows $P_1(x)+R_1(y)\geq x\cdot y$.

The proof of  Lemma 5.1$(ii)$ in the appendix shows $P_1(x_1,x_2,0)\geq\frac{1}{2}(x_1^2+x_2^2)$. Since $R_1(y)\leq R(y)$, we have
$$\int[\frac{1}{2}(y_1^2+y_2^2)-R(y)]d\nu(y)\leq\int[\frac{1}{2}(y_1^2+y_2^2)-R_1(y)]d\nu(y)$$
Put $h_0(x)=h_{P_0}(x)$, which is well defined since $-\frac{1}{\delta}\leq\partial_{x_3}P\leq-\delta$, then we have
$$\inf_{h\geq0}\int_{\Omega_{\infty}}[\frac{1}{2}(x_1^2+x_2^2)-P_0(x)]dx=\int_{\Omega_{h_0}}[\frac{1}{2}(x_1^2+x_2^2)-P_0(x)]dx=\int_{\Omega_{h_0}}[\frac{1}{2}(x_1^2+x_2^2)-\hat{P}_0(x)]dx$$
$$=\inf_{h\geq 0}\int_{\Omega_{h}}[\frac{1}{2}(x_1^2+x_2^2)-\hat{P}_0(x)]dx\leq\inf_{h\geq 0}\int_{\Omega_{h}}[\frac{1}{2}(x_1^2+x_2^2)-P_1(x)]dx$$
Since $P_1(x)\leq\hat{P}_0(x)$, which gives $J_{\nu}(P_0,R_0)=J_{\nu}(\hat{P}_0,R_0)\leq J_{\nu}(P_1,R_1)$. Summarizing above discussion, we get the following lemma.
\begin{lem}
Let $(P,R)\in\mathcal{N}_{\nu}$, then there exists $(P_1,R_1)\in\mathcal{N}_{\nu}$ such that

$(i)\,\,J_{\nu}(P,R)\leq J_{\nu}(P_1,R_1)$

$(ii)\,\,P_1(x_1,x_2,0)\geq\frac{1}{2}(x_1^2+x_2^2)\,\,\,and\,\,R_1(y)\leq\frac{1}{2}(y_1^2+y_2^2)$

$(iii)\,\,P_1(x),R_1(y)$ are convex conjugate over $\Omega_{\infty}$,$\Lambda$
\end{lem}
Let $(P_n,R_n)$ be a maximizing sequence of $J_{\nu}(P,R)$. By the above lemma, one can assume $(P_n,R_n)$ has additional properties $(ii),(iii)$ above. Besides, we can also assume $J_{\nu}(P_n,R_n)\geq0$ by (2.33). Next we will derive some estimates for the maximizing sequence, which allow us to pass to limit and prove the existence of at least one maximizer.
\begin{lem}
Let $(P_n,R_n)\in\mathcal{N}_{\nu}$ be a maximizing sequence of $J_{\nu}(P,R)$ with the properties $(ii),(iii)$ in previous lemma, Suppose $J_{\nu}(P_n,R_n)\geq0$,and put $h_n=h_{P_n}$, then there exists a constant $C=C(M_2(\nu),\delta,\Omega_2)$, such that
\begin{equation}
||h_n||_{L^2(\Omega_2)},||P_n(\cdot,0)||_{L^2(\Omega_2)}\leq C
\end{equation}
and constant $C_K=C(K,M_2(\nu),\delta,\Omega_2)$,such that
\begin{equation}
||P_n||_{L^2(\Omega_2\times[0,K])}\leq C_K
\end{equation}
We have also
\begin{equation}
-\max_{\Omega_2}|x||y|-\frac{2C_0}{\mathcal{L}^2(\Omega_2)}\leq R_n(y)\leq\frac{1}{2}(y_1^2+y_2^2)
\end{equation}
where $C_0$ has the same dependence as $C$ and in particular $||R_n||_{L^1(d\nu)}\leq C$
\end{lem}
\begin{proof}
We start with the following estimate
$$0\leq J_{\nu}(P_n,R_n)=\int[\frac{1}{2}(y_1^2+y_2^2)-R_n(y)]d\nu(y)+\inf_{h\geq0}\int[\frac{1}{2}(x_1^2+x_2^2)-P_n(x)]\sigma_h(x)dx$$
By taking $h\equiv\frac{2}{\mathcal{L}^2(\Omega_2)}$, and noticing that $-\frac{1}{\delta}\leq\partial_{x_3}P_n\leq-\delta$, we see
$$above \leq\int[\frac{1}{2}(y_1^2+y_2^2)-R_n(y)]d\nu(y)+\int_{\Omega_2\times[0,\frac{2}{\mathcal{L}^2(\Omega_2)}]}[\frac{1}{2}(x_1^2+x_2^2)-P_n(x)]dx$$
$$\leq\int[\frac{1}{2}(y_2^2+y_2^2)-R_n(y)]d\nu(y)+\int_{\Omega_2\times[0,\frac{2}{\mathcal{L}^2(\Omega_2)}]}[\frac{1}{2}(x_1^2+x_2^2)-P_n(x_1,x_2,0)+\frac{1}{\delta}x_3]dx$$
\begin{equation}
\leq\frac{1}{2}[M_2(\nu)+2\max_{\Omega_2}|x|^2+\frac{4}{\delta\mathcal{L}^2(\Omega_2)}]-\int R_n(y)d\nu(y)-\frac{2}{\mathcal{L}^2(\Omega_2)}\int_{\Omega_2}P_n(x_1,x_2,0)dx_1dx_2
\end{equation}
Now we notice
$$P_n(x_1,x_2,0)+R_n(y)\geq x_1y_1+x_2y_2\geq-\frac{1}{2}(x_1^2+x_2^2+y_2^2+y_2^2)$$
Therefore integrating above inequality against $\chi_{\Omega_2}\times d\nu$, we obtain
$$\int_{\Omega_2}P_n(x_1,x_2,0)dx_1dx_2+\mathcal{L}^2(\Omega_2)\int R_n(y)d\nu(y)\geq-\frac{1}{2}M_2(\nu)\mathcal{L}^2(\Omega_2)-\frac{1}{2}\mathcal{L}^2(\Omega_2)\max_{\Omega_2}|x|^2$$
So we deduce from (2.39)
$$0\leq\frac{1}{2}[M_2(\nu)+2\max_{\Omega_2}|x|^2+\frac{4}{\delta\mathcal{L}^2(\Omega_2)}]+\frac{1}{2}(M_2(\nu)+\max_{\Omega_2}|x|^2)-\frac{1}{\mathcal{L}^2(\Omega_2)}\int_{\Omega_2}P_n(x_1,x_2,0)$$
That is
\begin{equation}
\int_{\Omega_2}P_n(x_1,x_2,0)\leq\frac{\mathcal{L}^2(\Omega_2)}{2}[M_2(\nu)+2\max_{\Omega_2}|x|^2+\frac{4}{\delta\mathcal{L}^2(\Omega_2)}]+\frac{(\mathcal{L}^2(\Omega_2))}{2}(M_2(\nu)+\max_{\Omega_2}|x|^2):=C_0
\end{equation}
Next we want to derive a bound for $R_n$. By the $L^1$ bound derived above, we can find $(x_1^n,x_2^n)\in\Omega_2$,
such that $P_n(x_1^n,x_2^n,0)\leq\frac{2C_0}{\mathcal{L}^2(\Omega_2)}$,so
$$R_n(y)=\sup_{x\in\Omega_{\infty}}(x\cdot y-P_n(x))\geq x_1^ny_1+x_2^ny_2-P_n(x_1^n.x_2^n,0)$$
$$\geq-\max_{\Omega_2}|x||y|-\frac{2C_0}{\mathcal{L}^2(\Omega_2)}$$
To summarize
$$-\max_{\Omega_2}|x||y|-\frac{2C_0}{\mathcal{L}^2(\Omega_2)}\leq R_n(y)\leq\frac{1}{2}(y_1^2+y_2^2)$$
So we have
$$\int[\frac{1}{2}(y_1^2+y_2^2)-R_n(y)]d\nu(y)\leq\int[\frac{1}{2}(y_1^2+y_2^2)+\max_{\Omega_2}|x||y|+\frac{2C_0}{\mathcal{L}^2(\Omega_2)}]d\nu(y)$$
$$\leq M_2(\nu)+\max_{\Omega_2}|x|^2+\frac{2C_0}{\mathcal{L}^2(\Omega_2)}:=C_1$$
Next we proceed to derive an $L^2$ bound for $P_n(\cdot,0)$. Put $h_n=h_{P_n}$. This is well defined since $\partial_{x_3}P_n\leq-\delta$. Thus
$$-C_1\leq\inf_{h\geq0}\int[\frac{1}{2}(x_1^2+x_2^2)-P_n(x)]\sigma_h(x)dx=\int[\frac{1}{2}(x_1^2+x_2^2)-P_n(x)]\sigma_{h_n}(x)dx$$
To proceed further, we notice that since $P_n(x_1,x_2,h_n(x_1,x_2))=\frac{1}{2}(x_1^2+x_2^2)$, and for $0\leq x_3\leq h_n$, one has
$$\frac{1}{2}(x_1^2+x_2^2)\leq P_n(x)-\delta(h_n(x_1,x_2)-x_3)$$
So we conclude
$$-C_1\leq-\delta\int[h_n(x_1,x_2)-x_3]\sigma_{h_n}(x)dx=-\frac{\delta}{2}\int_{\Omega_2}h_n^2$$
That is
$$\int_{\Omega_2}h_n^2\leq\frac{2C_1}{\delta}:=C_2$$
Finally we notice
$$\frac{1}{2}(x_1^2+x_2^2)\leq P_n(x_1,x_2,0)\leq\frac{1}{\delta}h_n+P_n(x_1,x_2,h_n)=\frac{1}{\delta}h_n+\frac{1}{2}(x_1^2+x_2^2)$$
So we have $||P_n(\cdot,0)||_{L^2(\Omega_2)}\leq C_3$.

Since $-\frac{1}{\delta}x_3+P_n(x_1,x_2,0)\leq P_n(x)\leq P_n(x_1,x_2,0)-\delta x_3$, we have $||P_n||_{L^2(\Omega_2\times[0,K])}\leq C(K,M_2(\nu),\Omega_2)\,\,\forall\,\,n\geq1\,\,K\geq0$
\end{proof}
The bounds derived in previous lemma allow us to pass to limit and prove the existence of at least one maximizer. By Lemma 5.5 in the appendix, we know $P_n(\cdot,0)$ is uniformly bounded on each compact set $F$ of $\Omega_2$ and notice $-\frac{1}{\delta}\leq\partial_{x_3}P_n\leq-\delta$,$P_n$are uniformly bounded on each set of the form $F\times[0,K]$. Also we know they are equicontinuous on such a set and we can use Arzela-Ascoli to extract a subsequence which converges uniformly on each compact set of the form $F\times[0,K]$.

Therefore, we can take a subsequence(not relabeled)$P_n,R_n,h_n$, such that

$P_n\rightarrow P$ uniformly on each compact subset of $\Omega_{\infty}$ and in $L^r(\Omega_2\times[0,K])$, for any $r\in[1,2)$ and $K\geq0$.

$h_n\rightarrow \bar{h}$ locally uniformly in $\Omega_2$ and in $L^r(\Omega_2)$, for any $r\in[1,2)$\

$R_n\rightarrow R$ locally uniformly in $\Lambda$ and in $L^1(d\nu)$(by dominated convergence)

Since for each n, we have $P_n(x_1,x_2,h_n)=\frac{1}{2}(x_1^2+x_2^2)$, it`s easy to see in the limit $P(x_1,x_2,\bar{h})=\frac{1}{2}(x_1^2+x_2^2)$. This is by uniform convergence of $h_n,P_n$ on compact sets. So $\bar{h}=h_P$.

We now need to show the limit $(P,R)$ is a maximizer.
\begin{lem}
Let $P,R,\bar{h}$ be as in previous paragraph, then we have
$$J_{\nu}(P,R)\geq\lim_{n\rightarrow\infty}\sup J_{\nu}(P_n,R_n)$$
in particular, $(P,R)$ is a maximizer. Besides,
$$P(x_1,x_2,0)\geq\frac{1}{2}(x_1^2+x_2^2)$$
$$P(x)=\sup_{y\in\Lambda}(x\cdot y-R(y))$$
\end{lem}
\begin{proof}
First since $R_n$ converges in $L^1(d\nu)$, we have
$$\int[\frac{1}{2}(y_1^2+y_2^2)-R_n(y)]d\nu(y)\rightarrow\int[\frac{1}{2}(y_1^2+y_2^2)-R(y)]d\nu(y)$$
Next observe that on $\Omega_{h_n}$, we have $P_n(x)\geq\frac{1}{2}(x_1^2+x_2^2)$, so $0\leq P_n(x)\sigma_{h_n}(x)\rightarrow P(x)\sigma_{\bar{h}}(x)$ pointwise. So we can apply Fatou to see
$$\int P(x)\sigma_{\bar{h}}(x)\leq\lim_{n\rightarrow\infty}\inf\int P_n(x)\sigma_{h_n}(x)dx$$
Combining the $L^1$ convergence of $h_n$, we have
$$\int[\frac{1}{2}(x_1^2+x_2^2)-P(x)]\sigma_{\bar{h}}(x)dx\geq\lim_{n\rightarrow\infty}\sup\int[\frac{1}{2}(x_1^2+x_2^2)-P_n(x)]\sigma_{h_n}(x)dx$$
Therefore we see
$$J_{\nu}(P,R)\geq\lim_{n\rightarrow\infty}\sup J_{\nu}(P_n,R_n)$$
Since for each n, $P_n(x_1,x_2,0)\geq\frac{1}{2}(x_1^2+x_2^2)$, one easily sees this is preserved in the limit.

To see $P(x)=\sup_{y\in\Lambda}(x\cdot y-R(y))$, we first observe for each n,$P_n(x)=\sup_{y\in\Lambda}(x\cdot y-R_n(y))$
So $P_n(x)+R_n(y)\geq x\cdot y$, and the equality is achieved if $y\in\partial P_n(x)$. In the limit we have $P(x)+R(y)\geq x\cdot y$. Therefore
$$P(x)\geq\sup_{y\in\Lambda}(x\cdot y-R(y))$$

To see the reverse inequality, fix $x^0\in\Omega_{\infty}$. Find a compact set K such that $x^0\in K\subset\Omega_{\infty}$. Now $\partial P_n(K)$ must be uniformly bounded in $\Lambda$, by Lemma 5.5 in the appendix, so we can assume for some F compact $\partial P_n(K)\subset F\subset\Lambda$, for any n. Given $\epsilon>0$, by uniform convergence in compact sets we can take $n=n_0$, such that
$$|P(x)-P_n(x)|\leq\epsilon\,\,\,\forall\,\,x\in K\,\,\,and\,\,\,|R_n(y)-R(y)|\leq\epsilon\,\,\,\forall \,\,y\in F$$
Let $p\in\partial P_n(x^0)\subset F$, we then have
$$P(x^0)\leq P_n(x^0)+\epsilon=x^0\cdot p-R_n(p)+\epsilon\leq x^0\cdot p-R(p)+2\epsilon$$
$$\leq\sup_{y\in\Lambda}(x^0\cdot y-R(y))+2\epsilon$$
\end{proof}

On the other hand, if we define $\hat{R}(y)=\sup_{x\in\Omega_{\infty}}(x\cdot y-P(x))$, then
$\hat{R}(y)\leq  R(y)\leq\frac{1}{2}(y_1^2+y_2^2)$, and so $J_{\nu}(P,R)\leq J_{\nu}(P,\hat{R})$ and $(P,\tilde{R})$ is still a maximizer. Also because of the bound (2.36), (2.37) this bound will also be satisfied by $P$, so we can conclude $R$ satisfies the bound (2.38).
To summarize,we get
\begin{cor}
There exists a maximizer $(P_0,R_0)$ of $J_{\nu}(P,R)$,such that

$(i)\,\,(P_0,R_0)$ convex conjugate over $\Omega_{\infty}$ and $\Lambda$

$(ii)\,\,P_0(x_1,x_2,0)\geq\frac{1}{2}(x_1^2+x_2^2)$

$(iii)\,\,$Put $h_0=h_{P_0}$,then the bound in (2.36),(2.37),(2.38) hold true for $P_0,h_0,R_0$
\end{cor}
But the most useful properties will be the following
\begin{equation}
P(x)=\sup_{y\in\Lambda}(x\cdot y-R(y))\,\,\,x\in\Omega_{\infty}
\end{equation}
and
\begin{equation}
P(x_1,x_2,0)\geq\frac{1}{2}(x_1^2+x_2^2)
\end{equation}
The first condition ensures $-\frac{1}{\delta}\leq\partial_{x_3}P\leq-\delta$ so that $h_P$ is a well-defined function on $\Omega_2$. The second condition ensures $P(x_1,x_2,h)=\frac{1}{2}(x_1^2+x_2^2)$. Besides they can be preserved in the limit.

As before, we can prove the following property of maximizers. The argument is the same as Lemma 2.11, so we will omit the proof.
\begin{lem}
Let $(P,R)\in\mathcal{N}$ be a maximizer of $J_{\nu}(P,R)$ and such that (2.43),(2.44) are satisfied. Let $h(x)=h_P(x)$, then $\nabla P_{\sharp}\sigma_h=\nu$. In particular, $\int_{\Omega_2}h=1$
\end{lem}
Now we consider $E_{\nu}(h,\gamma)$. Similar to what we proved in previous subsection, we have the following.
\begin{lem}
Let $\nu\in\mathcal{P}_2(\mathbf{R}^3)$ with $supp\,\,\nu\subset\mathbf{R}^2\times[-\frac{1}{\delta},-\delta]$, then

$(i)\,\,E_{\nu}(h,\gamma)\geq J_{\nu}(P,R)$, for any $(h,\gamma)\in\mathcal{M}$ and $(P,R)\in\mathcal{N}$

$(ii)\,\,$Suppose $(P_0,R_0)$ be a maximizer of $J_{\nu}(P,R)$, and satisfies properties (2.43),(2.44), then $(i)$ takes equality if and only if $h=h_{P_0}$ and $\gamma=(id\times\nabla P_0)_{\sharp}\sigma_h$

$(iii)\,\,E_{\nu}(h,\gamma)$ has a unique minimizer $(h_0,\gamma_0)$ and there is a universal bound on h
\begin{equation}
||h_0||_{L^2(\Omega_2)}\leq C(M_2(\nu).\delta,\Omega_2)
\end{equation}

$(iv)\,\,$If $(P_0,R_0)$,$(P_1,R_1)$ are two maximizers of $J_{\nu}(P,R)$, both satisfy (2.43),(2.44), then we have $P_0=P_1$ on $\Omega_{h_0}\bigcup\{x_3=0\}$
\end{lem}
\begin{proof}
The proof for Theorem 2.18 and Corollary 2.19 works here. The $L^2$ bound on h is a consequence of Lemma 2.35. Also recall that by Corollary 2.36 $(ii)$whenever $h(x_1,x_2)=0$ for some $(x_1,x_2)\in\Omega_2$, we have exactly $P(x_1,x_2,0)=\frac{1}{2}(x_1^2+x_2^2)$.
\end{proof}
Next we study the stability property of these optimizers under narrow convergence of $\nu$ which remain bounded in $\mathcal{P}_2(\mathbf{R}_3)$.

Suppose $\nu^n,\nu\in\mathcal{P}_2(\mathbf{R}^3)$, $supp\,\,\nu^n,\nu\subset\Lambda$, and $\nu^n\rightarrow\nu$ narrowly with $\sup_{n\geq1}M_2(\nu^n)<\infty$. Let $(h^n,\gamma^n)$ be the minimizer of $E_{\nu^n}(h,\gamma)$, $(P^n,R^n)$ a maximizer of $J_{\nu^n}(P,R)$ given by Corollary 2.42. Then $\sup_{n\geq1}W_2(\nu^n,\nu)<\infty$. Now noticing $(iii)$ of Corollary 2.36,  , we can take a subsequence such that

$P^n\rightarrow \tilde{P}$ uniformly on each compact subset of $\Omega_{\infty}$ and in $L^r(\Omega_2\times[0,K])$, for any $r\in[1,2)$ and $K\geq0$

$h^n\rightarrow \tilde{h}$ locally uniformly in $\Omega_2$ and in $L^r(\Omega_2)$,for any $r\in[1,2)$

$R^n\rightarrow \tilde{R}$ locally uniformly in $\Lambda$

Put $\gamma^n=(id\times\nabla P^n)_{\sharp}\sigma_{h^n}$ and $\tilde{\gamma}=(id\times \nabla \tilde{P})_{\sharp}\sigma_{\tilde{h}}$, then above convergence implies $\gamma^n\rightarrow\tilde{\gamma}$ narrowly. Indeed
for any $g\in C_b(\mathbf{R}^3\times\mathbf{R}^3)$, we have
$$\int g(x,y)d\gamma^n(x,y)=\int g(x,\nabla P^n(x))\sigma_{h^n}(x)dx\rightarrow\int g(x,\nabla \tilde{P}(x))\sigma_{\tilde{h}}(x)dx$$$$=\int g(x,y)d\tilde{\gamma}$$
The above convergence is due to $L^1$ convergence of $h^n$ and point-wise convergence of $\nabla P^n$

Similar as before, the following semi-continuity result holds.
\begin{lem}
Let $h^n,\gamma^n,\tilde{h},\tilde{\gamma}$ be defined as in previous paragraph, then we have
$$E_{\nu}(\tilde{h},\tilde{\gamma})\leq\lim_{n\rightarrow\infty}\inf E_{\nu^n}(h^n,\gamma^n)$$
\end{lem}
\begin{proof}
This follows from the narrow convergence of $\gamma^n$ since the integrand is nonnegative.
\end{proof}
In the following lemma, we show that the limit $(\tilde{h},\tilde{\gamma})$ and $(\tilde{P},\tilde{R})$ obtained above are indeed optimizers of $E_{\nu}(h,\gamma)$ and $J_{\nu}(P,R)$ respectively.
\begin{lem}
The following statements are true.

$(i)\,\,$$\tilde{P}(x)=\sup_{y\in\Lambda}(x\cdot y-\tilde{R}(y))$, and $\tilde{P}(x_1,x_2,0)\geq\frac{1}{2}(x_1^2+x_2^2)$

$(ii)\,\,$$(\tilde{h},\tilde{\gamma})$ is the unique minimizer of $E_{\nu}(h,\gamma)$,

$(iii)\,\,$$(\tilde{P},\tilde{R})$ is a maximizer of $J_{\nu}(P,R)$.
\end{lem}
\begin{proof}
We first observe $\tilde{P}(x)=\sup_{y\in\Lambda}(x\cdot y-\tilde{R}(y))$. The argument is the same as Lemma 2.41.

To see $(\tilde{h},\tilde{\gamma})$ is the minimizer,  and $(\tilde{P},\tilde{R})$ is a maximizer, we prove it by showing
$E_{\nu}(\tilde{h},\tilde{\gamma})=J_{\nu}(\tilde{P},\tilde{R})$.
Indeed, for each $n$,
\begin{equation}
P^n(x)+R^n(\nabla P^n(x))=x\cdot\nabla P^n(x)\,\,\,\,\textrm{$\mathcal{L}^3$ a.e $x\in\Omega_{\infty}$}
\end{equation}
Passing both sides of (2.50) to limit, we then have
\begin{equation}
\tilde{P}(x)+\tilde{R}(\nabla \tilde{P}(x))=x\cdot\nabla\tilde{P}(x)\,\,\,\,\textrm{$\mathcal{L}^3$ a.e $x\in\Omega_{\infty}$}
\end{equation}
This means
\begin{equation}
\tilde{P}(x)+\tilde{R}(y)=x\cdot y\,\,\,\,\textrm{$\tilde{\gamma}$ a.e $(x,y)\in\Omega_{\infty}\times\Lambda$}
\end{equation}
Therefore
$$E_{\nu}(\tilde{h},\tilde{\gamma})=\int[\frac{1}{2}(x_1^2+x_2^2+y_1^2+y_2^2)-x\cdot y]d\tilde{\gamma}(x,y)$$
$$=\int[\frac{1}{2}(y_1^2+y_2^2)-\tilde{R}(y)]d\nu(y)+\int[\frac{1}{2}(x_1^2+x_2^2)-\tilde{P}(x)]\sigma_{\tilde{h}}(x)dx$$
Since $\tilde{h}=h_{\tilde{P}}$
$$\int[\frac{1}{2}(x_1^2+x_2^2)-\tilde{P}(x)]\sigma_{\tilde{h}}(x)dx=\inf_{h\geq0}\int[\frac{1}{2}(x_1^2+x_2^2)-\tilde{P}(x)]\sigma_h(x)dx$$
We obtain $E_{\nu}(\tilde{h},\tilde{\gamma})=J_{\nu}(\tilde{P},\tilde{R})$ \\
\end{proof}

We are now ready to prove the following stability result.
\begin{thm}
Let $\nu^n,\nu\in\mathcal{P}_2(\mathbf{R}^3)$, $supp\,\,\nu^n,\nu\subset\Lambda$, with $\nu^n\rightarrow\nu$ narrowly and $\sup_{n\geq1}M_2(\nu^n)<\infty$. Let $(h^n,\gamma^n)$,$(\tilde{h},\tilde{\gamma})$ be the unique minimizer of $E_{\nu}(h,\gamma)$, $E_{\nu}(h,\gamma)$ respectively, let $(P^n,R^n)$,$(\tilde{P},\tilde{R})$  be a maximizer of $J_{\nu^n}(P,R)$, $J_{\nu}(P,R)$ respectively which satisfies (2.43),(2.44), then the following convergence are true:

$(i)\,\,h^n\rightarrow \tilde{h}$ in $L^r(\Omega_2)$, for any $r\in[1,2)$.

$(ii)\,\,\gamma^n\rightarrow \tilde{\gamma}$ narrowly.

$(iii)\,\,\xi(P^n,\nabla P^n)\sigma_{h^n}\rightarrow\xi(\tilde{P},\nabla \tilde{P})\sigma_{\tilde{h}}$ in $L^1(\Omega_{\infty})$, for any $\xi\in C_b(\mathbf{R}\times\mathbf{R}^3)$

If one further assmes $W_2(\nu^n,\nu)\rightarrow0$, then $(i)$ can be improved to

$(iv)$$h^n\rightarrow \tilde{h}\in L^2(\Omega_2)$
\end{thm}
\begin{proof}
We can see from previous lemma that for any subsequence of $(h^n,\gamma^n)$, there is a further subsequence, say $(h^{n_j},\gamma^{n_j})$, such that $h^{n_j}$ converges in $L^r(\Omega_2)$, for any $r\in[1,2)$, $\gamma^{n_j}$ converges narrowly, and the limit is the unique minimizer of $E_{\nu}(h,\gamma)$. This is sufficient to conclude the whole sequence must converge. This proves $(i),(ii)$.

The argument for $(iii)$ is similar. We first show any sequence has a further subsequence which converges $\mathcal{L}^3-a.e$ to $\xi(\tilde{P},\nabla \tilde{P})\sigma_{\tilde{h}}$. Indeed, let $(\bar{P}^n,\bar{R}^n)$ be maximizers of $J_{\nu^n}(P,R)$ given by Corollary 2.42. Let $\xi(\bar{P}^n,\nabla \bar{P}^n)\sigma_{h^n}$ be a subsequence(not relabeled), by what has been discussed, we can take a further subsequence, say $(\bar{P}^{n_j},\bar{R}^{n_j})$ which converges locally uniformly on $\Omega_{\infty}$,$\Lambda$ respectively to a maximizer $(\bar{P},\bar{R})$, this maximizer will satisfy (2.43),(2.44), also $\xi(\bar{P}^{n_j},\nabla \bar{P}^{n_j})\sigma_{h^{n_j}}$ will converge $\mathcal{L}^3-a.e$ to $ \xi(\bar{P},\nabla\bar{ P})\sigma_{\tilde{h}}$, but $\xi(\bar{P}^n,\nabla \bar{P}^n)\sigma_{h^n}=\xi(P^n,\nabla P^n)\sigma_{h^n}$, $ \xi(\bar{P},\nabla\bar{ P})\sigma_{\tilde{h}}=\xi(\tilde{P},\nabla\tilde{ P})\sigma_{\tilde{h}}$  because of the uniqueness property proved in lemma 2.46 (iv).

To see such a convergence is in $L^1$, we just need to observe that $\sup_{n\geq1}\int_{\{x_3\geq K\}}\sigma_{h^n}dx\rightarrow0$ as $K\rightarrow\infty$. Indeed
$$\int_{\{x_3\geq K\}}\sigma_{h^n}dx=\int_{\Omega_2}(h_n-K)^+dx_1dx_2\leq\frac{||h_n||_{L^2(\Omega_2)}^2}{K}$$
Recall the universal bound on $||h^n||_{L^2(\Omega_2)}$ asserted in (2.47), we proved $(iii)$.

Now we assume $W_2(\nu^n,\nu)\rightarrow0$. Let $(\bar{P}^n,\bar{R}^n)$ be the above chosen maximizers of $J_{\nu^n}(P,R)$ given by Corollary 2.42 and $n_j$ is the above chosen subsequence which converges uniformly on compact sets. Because of the bound (2.38), and the assumed convergence $W_2(\nu^n,\nu)\rightarrow0$, we have
$$\int[\frac{1}{2}(y_1^2+y_2^2)-\bar{R}_{n_j}(y)]d\nu^n(y)\rightarrow\int[\frac{1}{2}(y_1^2+y_2^2)-\bar{R}(y)]d\nu(y)$$
while by Fatou(notice the integrand is nonpositive)
$$\int[\frac{1}{2}(x_1^2+x_2^2)-\bar{P}(x)]\sigma_{\tilde{h}}(x)dx\geq\lim_{n\rightarrow\infty}\sup\int[\frac{1}{2}(x_1^2+x_2^2)-\bar{P}^n(x)]\sigma_{h^n}(x)dx$$
Hence
$$J_{\nu}(\bar{P},\bar{R})\geq\lim_{n\rightarrow\infty}\sup J_{\nu^n}(\bar{P}^n,\bar{R}^n)$$
Recall Lemma 2.46$(i),(ii)$ and Lemma 2.48, we actually have under the stronger convergence of $\nu^n$
$$E_{\nu}(\tilde{h},\tilde{\gamma})=\lim_{n\rightarrow\infty}E_{\nu^n}(h^n,\gamma^n)$$
As a result of this,and noticing that $|x_1y_1+x_2y_2|\leq\frac{1}{2}(x_1^2+x_2^2+y_1^2+y_2^2)$,we have
$$\int(-y_3)x_3d\gamma^n(x,y)\rightarrow\int(-y_3)x_3\sigma_{\tilde{h}}d\tilde{\gamma}(x,y)$$
Recall $0\leq\delta x_3\leq(-y_3)x_3$,we obtain
$$\int x_3\sigma_{h^n}(x)dx\rightarrow\int x_3\sigma_{\tilde{h}}(x)dx$$
which implies
$$\int_{\Omega_2} (h^n)^2\rightarrow\int_{\Omega_2}(\tilde{h})^2$$
Combined with $(i)$, we get $L^2$ convergence of $h^n$
\end{proof}
To conclude this section, we observe that for the maximizers found for $J_{\nu}^H(P,R)$ in subsection 2.1 are also maximizers for $J_{\nu}(P,R)$, upon suitable extension. This fact is contained in the following lemma.
\begin{lem}
Let $\nu\in\mathcal{P}(\mathbf{R}^3)$ with $supp\,\,\nu\subset\Lambda$, where $\Lambda$ is as in (2.26). Let $H$ be chosen as (2.32). Let $(P,R)$ be the unique pair of maximizer given by Corollary 2.20. Define $\hat{P}(x)=\sup_{y\in\Lambda}(x\cdot y-R(y))\,\,\,x\in\Omega_{\infty}$ to be the extension of P to $\Omega_\infty$. Then

$(i)\,\,(\hat{P},R)$ are convex conjugate over $\Omega_{\infty},\Lambda$

$(ii)\,\,(\hat{P},R)$ is a maximizer for $J_{\nu}(P,R)$
\end{lem}
\begin{proof}
To prove $(i)$, we just need to check
$$R(y)=\sup_{x\in\Omega_{\infty}}(x\cdot y-\hat{P}(x))$$ This is routine.
To see $(\hat{P},R)$ is a maximizer for $J_{\nu}(P,R)$, we recall since $(P,R)$ is assumed to be a maximizer of $J_{\nu}^H(P,R)$, $\inf E_{\nu}(\gamma,h)=J_{\nu}^H(P,R)$ by Theorem 2.18. But we have $J_{\nu}^H(P,R)=J_{\nu}(\hat{P},R)$ because $h_{\tilde{P}}=h_{P}^H<H$ as guaranteed by the choice of sufficiently large H.
\end{proof}

\section{Existence of Lagrangian solutions}

In this section, we prove the existence of weak Lagrangian solutions when the initial dual density is absolutely continuous, using the properties of geostrphic functional $E_{\nu}(h,\gamma)$ already proved in subsection 2.1. The proof here is similar to \cite{Cullen-Feldman}.
\subsection{Basic definitions and main result}
First let`s define the notion of admissible initial data and weak Lagrangian solution.

 Fix $H_0,\delta>0$,
\begin{defn}
Given $(h_0,P_0)$ with $h_0\in W^{1,\infty}(\Omega_2)$ and $P_0\in W^{1,\infty}(\Omega_{H_0})$ with $H_0>\frac{\max_{x\in\Omega_2}P_0(x_1.x_2,0)}{\delta}$, we say it is an admissible initial data if the following holds true.

$(i)$$0\leq h_0\leq H_0$, $\int_{\Omega_2}h_0dx_1dx_2=1$,

$(ii)$$P_0$ is convex, and
$\partial P_0(\Omega_2\times[0,H_0])\subset \mathbf{R}^2\times[-\frac{1}{\delta},-\delta]$ and bounded.

$(iii)$$P(x_1,x_2,h_0(x_1,x_2))=\frac{1}{2}(x_1^2+x_2^2)$.

\end{defn}
\begin{rem}
Above assumption $H_0>\frac{\max_{x\in\Omega_2}P_0(x_1.x_2,0)}{\delta}$ actually guarantees $h_0<H_0$
\end{rem}
The above definition guarantees at least on $\Omega_{h_0}$, $P_0$ is the restriction of some maximizer of $J_{\nu_0}(P,R)$, where $\nu_0=\nabla P_{0\sharp}\sigma_{h_0}$. This is shown by the following proposition.
\begin{prop}
Let $D_0=||\nabla P_0||_{L^{\infty}}+1$. let $\Lambda_0=B_{D_0}\times[-\frac{1}{\delta},-\delta]$ , H as in (2.32). Let $(h_0,P_0)$ be admissible initial data. Let $\nu_0=\nabla P_{0\sharp}\sigma_{h_0}$ so that $supp\,\,\nu_0\subset \Lambda_0$, Let $(\bar{P}_0,\bar{R}_0)$ be the unique maximizer of $J_{\nu_0}^H(P,R)$ which satisfies the properties $(i),(ii)$ in Corollary 2.20, $(\tilde{h}_0,\gamma)$ be the unique minimizer of $E_{\nu_0}(h,\gamma)$, then
$P_0=\bar{P}_0$ on $\Omega_{h_0}\bigcup\{x_3=0\}$, and $\tilde{h}_0=h_0$,$\gamma=(id\times\nabla P_0)_{\sharp}\sigma_{h_0}$
\end{prop}
\begin{proof}
Define
 $$\hat{R}_0(y)=\sup_{x\in\Omega_{H}}(x\cdot y-P_0(x))$$
Then one has
$$P_0(x)+\hat{R}_0(y)\geq x\cdot y\,\,\,\,x\in\Omega_{H}\,\,y\in\Lambda$$
Also observe that
$$h_0(x_1,x_2)=h_{P_0}^H(x_1,x_2)$$
By $(iii)$ of above definition.\\
We just need to show $E_{\nu_0}(h_0,\gamma_0)=J_{\nu_0}^H(P_0,\hat{R}_0)$

Indeed, by definition
$$J_{\nu_0}^H(P_0,\hat{R}_0)=\int_{\Lambda}[\frac{1}{2}(y_1^2+y_2^2)-\hat{R}_0(y)]\nu_0(y)dy+\inf_{0\leq h\leq H}\int_{\Omega_{\infty}}[\frac{1}{2}(x_1^2+x_2^2)-P_0(x)]\sigma_h(x)dx$$
$$=\int_{\Lambda}[\frac{1}{2}(y_1^2+y_2^2)-\hat{R}_0(y)]\nu_0(y)dy+\int_{\Omega_{\infty}}[\frac{1}{2}(x_1^2+x_2^2)-P_0(x)]\sigma_{h_0}(x)dx$$
$$=\int_{\Omega_{\infty}\times\Lambda}[\frac{1}{2}(x_1^2+x_2^2)-P_0(x)+\frac{1}{2}(y_1^2+y_2^2)-\hat{R}_0(y)]d\gamma_0(x,y)$$

The second equality is due to above observation, and the third equality is because $\gamma_0$ has $\sigma_{h_0}$ and $\nu_0$ as marginals.

We will have shown $J_{\nu_0}^H(P_0,\hat{R}_0)=E_{\nu_0}(h_0,\gamma_0)$ provided we can show $P_0(x)+\hat{R}_0(y)=x\cdot y\,\,\gamma_0\,\,a.e$

Indeed, let`s show for $x_0\in\Omega_{h_0}$, such that $P_0$ is differentiable at $x_0$, we have
$$P_0(x_0)+\hat{R}_0(\nabla P_0(x_0))=x_0\cdot\nabla P_0(x_0)$$
This is implied by
$$P_0(x)\geq P_0(x_0)+\nabla P_0(x_0)(x-x_0)\,\,\,x\in\Omega_{\infty}$$
Since $P_0$ is convex in $\Omega_{\infty}$. So we have shown $J_{\nu}^H(P_0,\hat{R}_0)=E_{\nu_0}(h_0,\gamma_0)$.
So $(h_0,\gamma_0)$ defined above is the unique minimizer.

Now let $(\bar{P}_0,\bar{R}_0)$ be the unique maximizer of $J_{\nu_0}^H(P,R)$ given by corollary 2.20, then we have
$$\gamma_0=(id\times\nabla P_0)_{\sharp}\sigma_{h_0}=(id\times\nabla \bar{P}_0)_{\sharp}\sigma_{h_0}$$
and
$$h_0=h_{P_0}^H=h_{\bar{P}_0}^H$$
The same argument as in Corollary 2.19 implies $P_0=\bar{P}_0\,\,\,on\,\,\,\Omega_{h_0}$. Also for $(x_1,x_2)\in\Omega_2$ such that $h_0(x_1,x_2)=0$, we have $P(x_1,x_2,0)=\frac{1}{2}(x_1^2+x_2^2)$\\
\end{proof}
Next we define Eulerian solutions for the system in physical space.
\begin{defn}
Let $T>0$,$H>0$, Let $P(t,x)\in L^{\infty}([0,T],W^{1,\infty}(\Omega_H))$, such that
$$t\longmapsto P(t,\cdot)\,\,is \,\,convex\,\,\,\forall t\in[0,T]$$

Let $\mathbf{u}\in L^1([0,T]\times\Omega_{H})$.

Let $h\in L^{\infty}([0,T],W^{1,\infty}(\Omega_2))$be such that
 $$0\leq h<H\,\,\,\int_{\Omega_2}h=1$$

Then we say the triple $(P,\mathbf{u},h)$ is a weak Eulerian solution if the following is satisfied.\\
$(i)$$$\int_0^T\int_{\Omega_{h(t,\cdot)}}\mathbf{u}\cdot\nabla \phi(x,t)dxdt+\int_0^T\int
_{\Omega_{h(t,\cdot)}}\partial_t\phi(t,x)dxdt+\int_{\Omega_{h_0}}\phi(0,x)
dx=0$$
$\forall \phi\in C_c^1([0,T)\times\Omega_H)$\\

$(ii)$$P(t,x_1,x_2,h)=\frac{1}{2}(x_1^2+x_2^2)$ \\
$(iii)$$$\int_0^T\int_{\Omega_{h(t,\cdot)}}\nabla P(t,x)\partial_t\psi(t,x)dxdt+\int_0^T
\int_{\Omega_{h(t,\cdot)}}\nabla P(t,x)\cdot(\mathbf{u}\cdot\nabla)\psi(t,x)dxdt$$$$+\int_0^T\int_{\Omega_{h(t,\cdot)}}
[J(\nabla P(t,x)-x)]\cdot\psi(t,x)dxdt+\int_{\Omega_{h_0}}\nabla P_0(x)\psi(0,x)dx=0$$
$\forall \psi\in C_c^1([0,T)\times\Omega_H;\mathbf{R}^3)$.
\end{defn}

In the following we show that a weak Eulerian solution with sufficient regularity gives a classical solution. This justifies our definition.
\begin{prop}
Suppose $u\in C^1([0,T]\times\Omega_H)$, $h\in C^1([0,T]\times\Omega_2)$, $P\in C^2([0,T]\times\Omega_H)$ and $(P,\mathbf{u},h)$is a weak Eulerian solution, then they solve the equation in the classical sense.
\end{prop}
\begin{proof}
First we wish to deduce from $(i)$ above the divergence free of $\mathbf{u}$ and
the free boundary condition. Indeed under our regularity assumption
$$\int_0^T\int_{\Omega_h}\mathbf{u}\cdot\nabla \phi(t,x)dxdtdxdt$$
$$=\int_0^T\int_{\partial\Omega_h-\{x_3=h\}}(\mathbf{u}\cdot\mathbf{n})\phi(t,x)dSdt
+\int_0^T\int_{\{x_3=h\}}(\mathbf{u}\cdot\mathbf{n})\phi(t,x)dSdt$$$$-\int_0^T\int_{\Omega_h}(\nabla\cdot\mathbf{u})\phi(t,x)dxdt$$
The second term above can be written as
$$\int_0^T\int_{\{x_3=h\}}(\mathbf{u}\cdot\mathbf{n})\phi(t,x)dSdt=\int_0^T\int_{\Omega_2}(-u_1\partial_{x_1}h
-u_2\partial_{x_2}h+u_3)\phi(x,t)|_{x_3=h}dxdt$$
For the rest of the terms
$$\int_0^T\int
_{\Omega_h}\partial_t\phi(t,x)dxdt
dx$$$$=\int_0^T\int_{\Omega_2}(\partial_t\int_0^{h(t,x_1,x_2)}\phi dx_3)dt-\int_0^T\int_{\Omega_2}\partial_th\phi|_{x_3=h}dx_1dx_2dt$$
$$=-\int_{\Omega_{h_0}}\phi dx-\int_0^T\int_{\Omega_2}\partial_th\phi|_{x_3=h}dx_1dx_2dt$$
Combining terms to obtain
$$\int_0^T\int_{\partial\Omega_h-\{x_3=h\}}(\mathbf{u}\cdot\mathbf{n})\phi(t,x)dSdt
-\int_0^T\int_{\Omega_h}(\nabla\cdot\mathbf{u})\phi(t,x)dxdt$$$$+\int_0^T\int_{\Omega_2}(-\partial_th-u_1\partial_{x_1}-u_2\partial_{x_2}h+u_3)\phi(x,t)|_{x_3=h}dxdt=0$$
By choosing appriorate test funtions we see that
$$\nabla\cdot\mathbf{u}=0\,\,\,in\,\,\,\Omega_h$$As well as
$$\mathbf{u}\cdot\mathbf{n}=0\,\,on\,\,\partial_{\Omega_h}-\{x_3=h\}$$and
$$\partial_th+u_1\partial_{x_1}h+u_2\partial_{x_2}h=u_3\,\,\,on\,\,\,\{x_3=h\}$$
Finally we recover the equation $D_t(\nabla P)=J(\nabla P-x)$\\
First observe that
$$\int_0^T\int_{\Omega_h}\nabla P(t,x)\partial_t\psi(t,x)dxdt$$$$
=\int_{\Omega_2}dx_1dx_2\int_0^T(\partial_t\int_0^{h(t,x_1,x_2)}\nabla P(t,x)\psi(t,x)dx_3-\int_0^{h(t,x_1,x_2)}\partial_t(\nabla P)(t,x)\psi(t,x)dx_3)$$$$-\int_{\Omega_2}dx_1dx_2\int_0^T\nabla P(t,x)\psi(t,x)\partial_t h(t,x)|_{x_3=h}$$
$$=-\int_{\Omega_{h_0}}\int_0^T\nabla P_0(x)\psi(0,x)dx-\int_0^T\int_{\Omega_h}\partial_t(\nabla P)(t,x)\psi(t,x)dxdt-$$$$\int_{\Omega_2}dx_1dx_2\int_0^T\nabla P(t,x)\psi(t,x)\partial_t h(t,x)|_{x_3=h}$$
while the term
$$\int_0^T\int_{\Omega_h}\nabla P(t,x)(\mathbf{u}\cdot\nabla)\psi(t,x)dxdt$$$$=\int_0^T\int_{\Omega_h}\nabla
\cdot[(\nabla P\cdot\psi)\mathbf{u}]-(\mathbf{u}\cdot\nabla)(\nabla P(t,x))\cdot\psi(t,x)dxdt$$$$=\int_0^T\int_{\partial\Omega_h}(\nabla P\cdot\psi)\mathbf{u}\cdot\mathbf{n}dSdt-\int_0^T\int_{\Omega_h}(\mathbf{u}\cdot\nabla)(\nabla P(t,x))\cdot\psi(t,x)dxdt$$$$=\int_0^T\int_{\Omega_2}(-u_1\partial_{x_1}h-u_2\partial_{x_2}h+u_3)(\nabla P\cdot\psi)(t,x)|_{x_3=h}dxdt-\int_0^T\int_{\Omega_h}(\mathbf{u}\cdot\nabla)(\nabla P(t,x))\cdot\psi(t,x)dxdt$$
In the above ,we  used divergence free of $\mathbf{u}$ as well as the boundary condition $\mathbf{u}\cdot\mathbf{n}=0\,\,\,on\,\,\,\partial\Omega_h-\{x_3=h\}$

Now collect terms and use the free boundary condition to see the resulting equation is
exactly what we want.
\end{proof}
Now we define the notion of weak Lagrangian solution with admissible initial data.
\begin{defn}
Let $T>0$,$q\in(1,\infty)$, Let h be such that $$h(t,x_1,x_2)\in L^{\infty}([0,T),W^{1,\infty}(\Omega_2))\bigcap C^0([0,T]\times\bar{\Omega}_2)$$
Let $P(t,x_1,x_2,x_3)$ defined  on $\Omega_H$ be such that $$t\longmapsto P(t,\cdot)\,\,\, is \,\,\, convex\,\,\,\forall t\in[0,T]$$and additionally
$$P\in L^{\infty}([0,T],W^{1,\infty}(\Omega_H)\bigcap C([0,T]\times W^{1,r}(\Omega_H))$$
which assumes initial data $(h_0,P_0)$ on $\Omega_{h_0}\bigcup\{x_3=0\}$, i.e $h(0,x)=h_0(x)\,\,\,x\in\Omega_2$ and $P(0,x)=P_0(x)\,\,\,x\in\Omega_{h_0}\bigcup\{x_3=0\}$\\
Let $F:[0,T]\times\Omega_{h_0}\rightarrow \Omega_H$ be a Borel map and such that for some $1<r<\infty$.
$$F\in C([0,T];L^r(\sigma_{h_0}d\mathcal{L}^3,\mathbf{R}^3))$$
Then we say the triple $(h,P,F)$ is a weak Lagrangian solution with initial data$(h_0,P_0)$(admissible in the sense above) if the following holds:

$(i)$For each $t\in[0,T]$,$$0\leq h(t,x_1,x_2)<H\,\,\,and\,\,\,\int_{\Omega_2}h=1$$

$(ii)$$P(t,x_1,x_2,h(t,x_1,x_2))=\frac{1}{2}(x_1^2+x_2^2)$

$(iii)$$F(0,x)=x\,\,\sigma_{h_0}d\mathcal{L}^3\,\,a.e$ and $F_{t\sharp}\sigma_{h_0}=\sigma_{h_t}$

$(iv)$There exists Borel map $$F^*:\bigcup_{t\in[0,T]}\{t\}\times\Omega_{h(t,\cdot)}\rightarrow\Omega_H$$ such that $$F_t(F_t^*(x))=x\,\,\sigma_h\,\,a.e\,\,\,\,and \,\,\,F_t^*(F_t(x))=x\,\,\sigma_{h_0}\,\,a.e$$

$(v)$Put $$Z(t,x)=\nabla P_t(F_t(x))$$then the equation
$$\partial_tZ(t,x)=J(Z(t,x)-F(t,x))\,\,\,in\,\,\,\Omega_{h_0}$$is satisfied in the weak sense.i.e
$$\int_0^T\int_{\Omega_{h_0}}Z(t,x)\partial_t\phi(t,x)dxdt+\int_0^T\int_{\Omega_{h_0}}
J(Z(t,x)-F(t,x))\phi(t,x)dxdt$$$$+\int_{\Omega_{h_0}}Z(0,x)\phi(0,x)dx=0$$
$\forall \phi(t,x)\in C_c^1([0,T)\times \Omega_H)$
\end{defn}
\begin{rem}
$(i)$We only assume $P$ achieves the initial data on $\Omega_{h_0}\bigcup\{x_3=0\}$ instead of $\Omega_H$ because only the region $\Omega_{h_0}$ is physically relevant.

$(ii)$In the $(v)$ of above definition, it is possible to choose test function $\phi$ such that $\partial_t\phi\in L^{\infty}([0,T)\times\Omega_H)$ without assuming $\nabla_x\phi$ exists. Indeed, we may define $\phi(-t,x)\equiv\phi(0,x)$ for $t>0$ and convolve with $J_s(t,x)$ with s small such that $\phi_s:=\phi*j_s$ is a legitimate test function. Then $\partial_t\phi_s,\phi_s$ converges a.e in $[0,T]\times\Omega_{\infty}$ to $\partial_t\phi,\phi$ as $s\rightarrow0$ then the result follows from dominated convergence.

$(iii)$In the $(iv)$ above, note that
$$\bigcup_{t\in[0,T]}\{t\}\times\Omega_{h(t,\cdot)}=\{(t,x)\in[0,T]\times\Omega_H|0<x_3<h(t,x_1,x_2)\}$$
\end{rem}
Of course we must show above definition does not lose any information, i.e we need to show with additional regularity assumption, weak Lagrangian solution gives weak Eulerian
solution.\\
\begin{prop}
Suppose $(h,P,F)$ is a weak Lagrangian solution, suppose also that $\partial_t F_t\in L^{\infty}([0,T]\times\Omega_H)$, Define $\mathbf{u}(t,x)=\partial_tF_t(F_t^*(x))$, then $(P,\mathbf{u},h)$ gives a weak Eulerian solution.
\end{prop}
\begin{proof}
We only need to check that $(i)$ and $(iii)$ of the definition of Eulerian solution are satisfied.\\First notice that $F_{t\sharp}\sigma_{h_0}=\sigma_{h}$. Thus for $\phi\in
C_c^1([0,T)\times\Omega_{\infty})$, we have
$$\int_0^T\int_{\Omega_h}\mathbf{u}\cdot\nabla\phi(x,t)dxdt
=\int_0^T\int_{\Omega_h}\partial_tF_t(F_t^*(x))\nabla\phi(x,t)dxdt$$
$$=\int_0^T\int_{\Omega_{h_0}}\partial_tF_t(x)\nabla\phi(F_t(x),t)dxdt$$
$$=\int_0^T\int_{\Omega_{h_0}}\partial_t(\phi(F_t(x),t))-\partial_t\phi(F_t(x),t)dxdt$$
$$=-\int_{\Omega_{h_0}}\phi(x,0)dx-\int_0^T\int_{\Omega_{h_0}}\partial_t\phi(F_t(x),t)
dxdt$$$$=-\int_{\Omega_{h_0}}\phi(x,0)dx-
\int_0^T\int_{\Omega_{h}}\partial_t\phi(x,t)dxdt$$
This verifies $(i)$.\\In the second line above, we used $F_{t\sharp}\sigma_{h_0}=\sigma_{h
(t,\cdot)}$ and also $F_t^*(F_t(x))=x\,\,\sigma_{h_0}\,\,a.e$
In the third line above, we used $$for\,\,\sigma_{h_0}\,\,a.e\,\,x\,\\,\,\,\partial_t(\phi(F_t(x),t))=\partial_t\phi(F_t(x),t)+\nabla\phi(F_t(x),t)\partial
_tF_t(x)$$since $t\longmapsto F_t(x)$ is Lipschitz.

Now we verify $(iii)$. Indeed
$$\int_0^T\int_{\Omega_h}\nabla P(t,x)\partial_t\psi(t,x)dxdt=
\int_0^T\int_{\Omega_{h_0}}Z(t,x)\partial_t\psi(t,F_t(x))dxdt$$
$$=\int_0^T\int_{\Omega_{h_0}}Z(t,x)[\partial_t(\psi(t,F_t(x))-\nabla\psi(t,F_t(x))\partial_tF_t(x)]dxdt$$
$$=\int_0^T\int_{\Omega_{h_0}}Z(t,x)\partial_t(\psi(t,F_t(x)))dxdt-\int_0^T\int_{\Omega_h}
(\mathbf{u}\cdot\nabla )\psi(t,x)\nabla P(t,x)dxdt$$
In the second line above, we used the fact that since $\partial_tF_t(x)\in L^{\infty}([0,T]\times\Omega_H)$,$\partial_t(\psi(t,F_t(x)))\in L^{\infty}([0,T]\times\Omega_H)$ and usual chain rule holds.\\
Now we  choose test function as $\psi(t,F_t(x))$ in the definition$(v)$ above to get the following. This is justified because of the remark after the definition above.
$$\int_0^T\int_{\Omega_{h_0}}Z(t,x)\partial_t(\psi(t,F_t(x))dxdt$$$$=
-\int_0^T\int_{\Omega_{h_0}}J(Z(t,x)-F(t,x))\psi(t,F_t(x))dxdt-\int_{\Omega_{h_0}}Z(0,x)\psi(0,x)dx$$
$$=-\int_0^T\int_{\Omega_h}J(\nabla P(t,x)-x)\psi(t,x)dxdt-\int_{\Omega_{h_0}}\nabla P_0(x)\psi(0,x)dx$$
Put things together, we get $(iii)$.
\end{proof}
Now we can state the existence result of weak Lagrangian solutions.
\begin{thm}
Let $T>0$,$,1<q<\infty$, and admissible initial data $(h_0,P_0)$ be given, suppose also $$\nu_0:=\nabla P_{0\sharp}\sigma_{h_0}\in L^q(\mathbf{R}^3)\,\,$$
Suppose also that
\begin{equation}
H>\frac{2}{\mathcal{L}^2(\Omega_2)}+\frac{2diam\,\,\Omega_2}{\delta}[||\nabla P_0||_{L^{\infty}}T+\max_{\Omega_2}|x|(T+2)+2]
\end{equation}
then there exists a weak Lagrangian solution $(h,P,F)$ on $[0,T]\times\Omega_H$. Moreover the function $Z(\cdot,x)\in W^{1,\infty}(\mathbf{R}^3)$ for a.e $x\in\Omega_{h_0}$ and the equations are satisfied in the following sense:
$$\partial_tZ(t,x)=J(Z(t,x)-F(t,x))\,\,\,\,for\,\,\,\mathcal{L}^4-a.e\,\,\,in\,\,\,(t,x)\in[0,T]\times\Omega_{h_0}$$
$$Z(0,x)=\nabla P_0(x)\,\,\,for\,\,\,\mathcal{L}^3-a.e\,\,in\,\,x\in\Omega_{h_0}$$
\end{thm}

\subsection{Lagrangian flow in dual space}
Next we study the Lagrangian flow in dual space. It is similar to \cite{Cullen-Feldman} section 2.3\\
Recall  that $\nu_0=\nabla P_{0\sharp}\sigma_{h_0}$,so
\begin{equation}
supp\,\,\nu_0\subset B_{D_0}(0)\times[-\frac{1}{\delta},-\delta]:=\Lambda_0
\end{equation}
where
\begin{equation}
D_0={||\nabla P_0||_{L^{\infty}}+1}
\end{equation}
Put
\begin{equation}
D=D_0+T\max_{\Omega_2}|x|+1
\end{equation}
 and define
\begin{equation}
\Lambda=B_D(0)\times[-\frac{1}{\delta},-\delta]
\end{equation}
Choose
$H$ as in (3.10).
Let $h,P,R,\nu$ be the dual space solution given by Theorem 2.30 where we have chosen the parameters $T,H,D,D_0$ as here, recall that $\nu(t,\cdot)$ satisfies the transport equation
$$\partial_t\nu+\mathbf{w}\cdot\nabla\nu=0\,\,\,\nu(0,\cdot)=\nu_0$$
 where $\mathbf{w}=J(y-\nabla R(t,y))$ is divergence free.

Then it follows from Theorem 2.30 that
$$\mathbf{w}\in L_{loc}^{\infty}([0,T]\times\mathbf{R}^3)\,\,\,\mathbf{w}\in L^{\infty}([0,T],BV(B(0,R)))\,\,\forall R>0$$
Here we naturally extend $\mathbf{w}$ to $\mathbf{R}^3$ by the same formula above.

Since $\nu$ is  supported in $B_D(0)\times[-\frac{1}{\delta},-\delta]\times[0,T]$
 we can modify $\mathbf{w}$ outside $B_D(0)\times[-\frac{1}{\delta},-\delta]$ such that the modified $\tilde{\mathbf{w}}\in L^{\infty}([0,T]\times\mathbf{R}^3)$ and $\tilde{\mathbf{w}}\in L^{\infty}([0,T],BV(B(0,R)))\,\,\,\forall R>0$ and $\nabla\cdot( \tilde{\mathbf{w}}(t,\cdot))=0$.
To construct $\tilde{w}$, let $\zeta\in C^{\infty}(\mathbf{R}^3)$ be a cut-off function such that
$$\zeta(s)\equiv1\,\,for\,\,|s|\leq D+\frac{1}{\delta}\,\,\,\,\zeta(s)\equiv0\,\,for\,\,|s|\geq D+\frac{1}{\delta}+1\,\,\,0\leq\zeta\leq1$$
Then define
$$H(y)=(\zeta(|y_1|)y_1,\zeta(|y_2|)y_2,\zeta(|y_3|)y_3)$$
Then we take
$$\tilde{\mathbf{w}}=J(H(y)-\nabla R(t,y))\,\,\,y\in\mathbf{R}^3$$
Now we can apply Ambrosio`s theory \cite{BV vector} to $\tilde{\mathbf{w}}$ to get a Lagrangian flow $\Phi$ in dual space and establish the following lemma.
\begin{lem}
Let $\tilde{\mathbf{w}}$ be defined as above. Then there exists a unique locally bounded Borel measurable map $\Phi:[0,T]\times\mathbf{R}^3\rightarrow\mathbf{R}^3$, satisfying

$(i)$$\Phi(\cdot,y)\in W^{1,\infty}([0,T];\mathbf{R}^3)$ for a.e $y\in\mathbf{R}^3$

$(ii)$$\Phi(0,y)=y$ for $\mathcal{L}^3$-a.e $y\in\mathbf{R}^3$

$(iii)$For a.e $(t,y)\in\mathbf{R}^3\times(0,T)$
$$\partial_t\Phi(t,y)=\tilde{\mathbf{w}}(t,\Phi(t,y))$$

$(iv)$ $\Phi(t,\cdot):\mathbf{R}^3\rightarrow\mathbf{R}^3$ is a $\mathcal{L}^3$ measure preserving map for every $t\in[0,T]$
\end{lem}
\begin{rem}
Notice that by our definition of $\tilde{\mathbf{w}}$, we have $\tilde{\mathbf{w}}_3=0$. By $(i)$ and $(iii)$ above, one sees that
$$\Phi_3(t,y)\equiv y\,\,\,\forall t\in[0,T]$$

for $\mathcal{L}^3-a.e \,\,y$
\end{rem}
\begin{lem}
Let $\tilde{\mathbf{w}}$ be defined as above, and let $\Phi$ be the flow in previous lemma, then\\
$(i)$
$$\Phi(t,y)\subset B^2_D(0))\times[-\frac{1}{\delta},-\delta]\,\,\,for\,\,\,a.e\,\,(t,y)\in[0,T]\times\nabla P_0(\Omega_{\infty})$$
In particular,
$$\partial_t\Phi(t,y)=\mathbf{w}(t,\Phi(t,y))\,\,\,for\,\,\,a.e\,\,(t,y)\in[0,T]\times\nabla P_0(\Omega_{\infty})$$

$(ii)$There exists a Borel map $\Phi^*:[0,T]\times\mathbf{R}^3\rightarrow\mathbf{R}^3$ such that for every $t\in[0,T]$ the map $\Phi_t^*:\mathbf{R}^3\rightarrow\mathbf{R}^3$ is $\mathcal{L}^3$ measure preserving, and such that $\Phi_t^*\circ\Phi_t(y)=y$ and $\Phi_t\circ\Phi_t^*(y)=y$ for a.e $y\in\mathbf{R}^3$
\end{lem}
The proof of this lemma is almost identical to the lemma 2.9 of\cite{Cullen-Feldman}, so we omit the proof.

\begin{prop}
Let $\Omega_H$,T,q be as in Theorem 2.30. Let $h,P,R,\nu$ be the dual space solution obtained in that theorem. Let $\tilde{\mathbf{w}},\Phi$ be as in above definition. Let $D=||\nabla P_0||_{L^{\infty}}+(T+1)\max_{\Omega_2}|x|+1$, then for $t\in[0,T]$
$$\nu(t,\cdot)=\Phi_{t\sharp}\nu_0$$
Moreover, for $t\in[0,T]$, we have
$$\nu(t,y)=\nu_0(\Phi_t^*(y))\,\,\,a.e\,\, y$$
\end{prop}
The proof is again similar to Proposition 2.11\cite{Cullen-Feldman}. Keep in mind that we have strong convergence of the transporting vector $\mathbf{w}_{s_j}$ to $\mathbf{w}$ according to the proof of Theorem 2.30. This enables us to use the stability result Theorem 5.5 of \cite{BV vector}

\subsection{Lagrangian flow in physical space}

We want to define a Lagrangian flow in the physical space $F:[0,T]\times\Omega_{h_0}\rightarrow\Omega_H$ by
defining $F_t:\Omega_{h_0}\rightarrow\Omega_H$ by
$$F_t:=\nabla R_t\circ\Phi_t\circ\nabla P_0$$
Of course, we need to check above formula is well defined.
\begin{lem}
The right hand side above is defined $\mathcal{L}^4-a.e$. For any $t\in[0,T]$, above is defined $\mathcal{L}^3-a.e.$
and $F:[0,T]\times\Omega_{h_0}\rightarrow\Omega_H$ is a Borel map.
\end{lem}
This is checked in the same way as Lemma 2.12 of \cite{Cullen-Feldman}. So we omit the proof.\\
In the following, we collect some results which can be proved in the same way as \cite{Cullen-Feldman}, they correspond to Proposition 2.13-2.16 in that paper.
\begin{lem}
$$F(0,x)=x\,\,a.e\,\,x\in\Omega_{h_0}$$
\end{lem}

\begin{lem}
For every $t\in[0,T]$, we have $$F_{t\sharp}\sigma_{h_0}=\sigma_{h(t,\cdot)}$$

\end{lem}

\begin{prop}
For any $t_0\in[0,T]$,any $r\in[1,\infty)$, one has
$$\lim_{t\rightarrow t_0,t\in[0,T]}\int_{\Omega_{h_0}}|F_t(x)-F_{t_0}(x)|^rdx=0$$
\end{prop}

\begin{prop}
There exists a Borel map $$F^*:\bigcup_{t\in[0,T]}\{t\}\times\Omega_{h(t,\cdot)}\rightarrow\Omega_H$$
such that for any $t\in[0,T]$
$$F_t(F_t^*(x))=x\,\,\sigma_{h(t,\cdot)}-a.e\,\,\, and \,\,\, F_t^*(F_t(x))=x\,\,\sigma_{h_0}-a.e$$
\end{prop}

\begin{prop}
Put $Z(t,x)=\nabla P_t(F_t(x))$, then $Z(t,x)$ satisfies
$$\int_0^T\int_{\Omega_{h_0}}Z(t,x)\partial_t\phi(t,x)dxdt+\int_0^T\int_{\Omega_{h_0}}
J(Z(t,x)-F(t,x))\phi(t,x)dxdt$$$$+\int_{\Omega_{h_0}}Z(0,x)\phi(0,x)dx=0$$
$\forall \phi(t,x)\in C_c^1([0,T)\times \Omega_H)$ \\
Moreover, possibly after changing $Z(t,x)$ on a negligible subset of $[0,T]\times\Omega_{h_0}$, we have $Z(\cdot,x)\in W^{1,\infty}([0,T];\mathbf{R}^3)$ for $a.e-x\in\Omega_{h_0}$ and
$$\partial_tZ(t,x)=J(Z(t,x)-F(t,x))\,\,\,\,for\,\,\,\mathcal{L}^4-a.e\,\,\,in\,\,\,(t,x)\in[0,T]\times\Omega_{h_0}$$
$$Z(0,x)=\nabla P_0(x)\,\,\,for\,\,\,\mathcal{L}^3-a.e\,\,in\,\,x\in\Omega_{h_0}$$
\end{prop}
\begin{proof}
By Lemma 3.20, we have
$$Z(t,x)=\nabla P_t\circ F_t(x)=\Phi_t\circ \nabla P_0(x)$$
Except for a $\mathcal{L}^4$ negligible set.

So after redefining $Z$ on a negligible set, we may redefine
$Z(t,x)=\Phi_t\circ\nabla P_0(x)$ and we will prove this version of Z has all the properties claimed.

By  Lemma 3.16, we know that $\Phi(\cdot,y)\in W^{1,\infty}([0,T];\mathbf{R}^3)\,\,for\,\,a.e\,\,y$. Since $\nabla P_{0\sharp}\sigma_{h_0}<<\mathcal{L}^3$, one has
$$t\longmapsto \Phi_t(\nabla P_0(x))\in W^{1,\infty}([0,T];\mathbf{R}^3)\,\,\,for\,\,\mathcal{L}^3-a.e\,\,x\in\Omega_{h_0}$$
Let $\tilde{N}_3$ be such that $\mathcal{L}^3(\tilde{N}_3)=0$ and for $y\in\tilde{N}_3^c$,$(i),(ii)$ of Lemma 3.16 holds, and such that for $a.e\,\,t\in [0,T]$, $\partial_t\Phi(t,y)=\tilde{\mathbf{w}}(t,\Phi(t,y))$ holds. Then for such y, one has
$$\Phi(t,y)=y+\int_0^t\tilde{\mathbf{w}}(s,\Phi(s,y))ds\,\,\,\forall t\in[0,T]$$
Due to the same reason as above and note that $|\nabla P_0(x)|\leq D\,\,\forall x\in\Omega_{h_0}$ by our choice of D, one can conclude for $\mathcal{L}^3-a.e\,\, x\in\Omega_{h_0}$
$$\Phi(t,\nabla P_0(x))=\nabla P_0(x)+\int_0^t\mathbf{w}(s,\Phi(s,\nabla P_0(x))ds\,\,\,\forall t\in[0,T]$$
Therefore $$Z(0,x)=x\,\,\,\mathcal{L}^3-a.e\,\,x\in\Omega_{h_0}$$
Also for such $x$, one has
$$\partial_t\Phi(t,\nabla P_0(x))=\mathbf{w}(t,\Phi(t,\nabla P_0(x)))$$
$$=J(\Phi(t,\nabla P_0(x))-\nabla R_t\circ\Phi_t\circ\nabla P_0(x))$$
$$=J(Z(t,x)-F(t,x))\,\,a.e\,\,t\in[0,T]$$
So we have shown now
$$\partial_tZ(t,x)=J(Z(t,x)-F(t,x))\,\,\,\,for\,\,\,\mathcal{L}^4-a.e\,\,\,in\,\,\,(t,x)\in[0,T]\times\Omega_{h_0}$$

Next we notice since we assumed $P_0$ to be Lipschitz on $\Omega_H$. $\nabla P_0(\Omega_H)$ will remain in a bounded set. And the flow map $\Phi$ is locally bounded on $\mathbf{R}^3\times[0,T]$, therefore $Z(t,x)=\Phi_t\circ\nabla P_0(x)$ will be bounded for $\Omega_H\times[0,T]$. So we can multiply a test function to above $a.e$ identity and integrate by parts in $t$ to get the distributional identity.
\end{proof}
$\mathbf{Proof\,\,of\,\,theorem\,\,3.9}$:

Let $D_0=||\nabla P_0||_{L^{\infty}}+1$, set $\Lambda_0=B_{D_0}(0)\times[-\frac{1}{2\delta},-\frac{\delta}{2}]$, and put $\nu_0:=\nabla P_{0\sharp}\sigma_{h_0}$. Then $supp\,\,\nu_0\subset \Lambda_0$. Also put $D=D_0+(T+1)\max_{\Omega_2}|x|+1$ and $\Lambda=B_D(0)\times[-\frac{1}{\delta},-\delta]$.

Choose $H$ as stated in Theorem 2.3. Then such H also satisfies
$$H>\frac{2}{\mathcal{L}^2(\Omega_2)}+\frac{2\max_{\Omega_2}|x|+2D}{\delta}\cdot diam\,\,\Omega_2$$
as required by Theorem 2.30.

Now let $(h,P,R)$ be the dual space  solution corresponding to the initial data $\nu_0$ given by Theorem 2.30, and let $F$ be as defined as in the begining of this subsection. Then $(h,P,F)$ gives a weak Lagrangian solution with initial data $(h_0,P_0)$.\\
Indeed, that $(h,P)$ assumes initial data $(h_0,P_0)$ on $\Omega_{h_0}\bigcup\{x_3=0\}$ is guaranteed by Proposition 3.3. property $(i),(ii)$ comes from our construction of dual space solution, other properties follows from the lemmas and propositions listed in this subsection.

\section{Existence of relaxed Lagrangian solution with general initial data($\nabla P_0\in L^2$)}
\subsection{Definition of generalized data and main result}
In this section, we prove the existence of relaxed Lagrangian solution in a similar way as was done in \cite{Feldman-Tudorascu}.

 The relaxed Lagrangian solution is an even weaker notion than weak Lagrangian solutions defined in previous sections, but it will allow for more general initial data $(h_0,P_0)$, in particular, we will no longer require $\nabla P_{0\sharp}\sigma_{h_0}<<\mathcal{L}^3$ or have compact support (still we assume $supp\,\,\subset\{-\frac{1}{\delta}\leq x_3\leq-\delta\}$). To motivate the definition, recall that the weak Lagrangian solution given by Theorem  3.10 will satisfy the additional property
$$\partial_tZ(t,x)=J(Z(t,x)-F(t,x))\,\,\,\,for\,\,\,\mathcal{L}^4-a.e\,\,\,in\,\,\,(t,x)\in[0,T]\times\Omega_{h_0}$$
$$Z(0,x)=\nabla P_0(x)\,\,\,for\,\,\,\mathcal{L}^3-a.e\,\,in\,\,x\in\Omega_{h_0}$$

Recalling the definition of $Z(t,x)$, this implies for $\xi\in C^1(\mathbf{R}^3)$, we have
$$\partial_t\xi(\nabla P_t(F_t(x))=\nabla\xi(\nabla P_t(F_t(x))\cdot J(\nabla P_t(F_t(x))-F(t,x))\,\,\,\,for\,\,\,\mathcal{L}^4-a.e\,\,\,in\,\,\,(t,x)\in[0,T]\times\Omega_{h_0}$$
and has initial data $\xi(\nabla P_0(x))$.

Thus if we define a Borel family of  measures $[0,T]\ni t\longmapsto d\alpha_t(x,\hat{x}):=(id\times F_t)_{\sharp}\sigma_{h_0}$ and put $d\alpha=d\alpha_tdt$, then at least formally we can obtain from above almost everywhere defined equality that
\begin{equation}
\int_0^T\int_{\Omega_{\infty}^2}\xi(\nabla P_t(\hat{x}))\partial_t\psi(t,x)d\alpha_t(x,\hat{x})dt+\int_0^T\int_{\Omega_{\infty}^2}\nabla \xi(\nabla P_t(\hat{x}))\cdot J(\nabla P_t(\hat{x})-\hat{x})\psi(t,x)d\alpha_t(x,\hat{x})dt
\end{equation}
$$+\int_{\Omega_{h_0}}\xi(\nabla P_0(x))\psi(0,x)dx=0\,\,\,\forall \xi\in C_c^1(\mathbf{R}^3)$$
Above discussion motivates the following definition of relaxed Lagrangian solutions.
\begin{defn}
$(h_0,P_0)$ be admissible initial data in the sense of Definition 4.3 below. Consider  a Borel function $P:\Omega_{\infty}\times[0,T]\rightarrow\mathbf{R}$, such that $P(t,\cdot)$ is convex for any $t\in[0,T]$ and $h:\Omega_2\times[0,T]\rightarrow\mathbf{R}$ such that $h(t,\cdot)\in L^2(\Omega_2)$  and a family of Borel measures $[0,T]\ni t\longmapsto \alpha_t(x,\hat{x})\in\mathcal{P}(\Omega_{\infty}\times\Omega_{\infty})$. Let $\alpha$ be given by $d\alpha=d\alpha_tdt$. We say that the triple $(P,h,\alpha)$ is a relaxed Lagrangian solution if the following holds

$(i)$$\,\,h\geq0\,\,\,\int_{\Omega_2}hdx_1dx_2=1\,\,\,\forall t\in[0,T]$

$(ii)$$\,\, P_t(x_1,x_2,h)=\frac{1}{2}(x_1^2+x_2^2)$,$\nabla P_t\in L^2(\Omega_{h_t})$, and $-\frac{1}{\delta}\leq\partial_{x_3}P_t\leq-\delta$ for any $t\in[0,T]$

$(iii)$$proj_x\alpha_t=\sigma_{h_0}\,\,\,\,\,proj_{\hat{x}}\alpha_t=\sigma_{h(t,\cdot)}$, $\alpha_0=(id\times id)_{\sharp}\sigma_{h_0}$

$(iv)$(4.1) above holds true.

$(v)$$P(0,\cdot)=P_0\,\,\,on\,\,\,\Omega_{h_0}\,\,\,h(0,\cdot)=h_0$
\end{defn}
Since we derived (4.1) only formally, we need to check (4.1) makes sense.
\begin{lem}
The left hand side of (4.1) is well-defined for any $d\alpha(t,x,\hat{x})=d\alpha_t(x,\hat{x})dt$ with $\alpha_t$ satisfying $(iii)$ in Definition 4.2  .
\end{lem}
\begin{proof}
The main issue comes from the fact that for each fixed t, $\nabla P_t$ is not an honest function but is defined only $a.e$, We have to check that different choice of Borel representative of $\nabla P_t$ does not affect the integral in the left hand side. Actually we will show for each fixed time slice t, the inner integral is well-defined.

Fix $t\in[0,T]$, let $\Psi_1,\Psi_2$ be two Borel representatives of $\nabla P_t$ ,i.e $\Psi_1,\Psi_2$ are Borel functions and $\Psi_i=\nabla P_t\,\,\mathcal{L}^3-a.e$. Concerning the first term, consider the set
$$E=\{(x,\hat{x})\in\Omega_{\infty}^2|\xi(\Psi_1(\hat{x}))\partial_t\psi(t,x)\neq \xi(\Psi_2(\hat{x}))\partial_t\psi(t,x)\}$$

Then $\alpha_t(E)=\sigma_h(\{\hat{x}\in\Omega_{\infty}|\Psi_1(\hat{x}))\neq\Psi_2(\hat{x}))\})=0$. The argument for other terms are the same.
\end{proof}
Now we define a more general class of initial data, and state the existence result of renormalized solutions with such data. To prove the existence, we need to study the functional $E_{\nu}(h,\gamma)$ and $J_{\nu}(P,R)$ with more general $\nu$ than considered in section 3. This is done in subsection 2.2. As a byproduct of our proof, we will get the dual space existence result with a general measure valued initial data whose support is contained in $\mathbf{R}^2\times[-\frac{1}{\delta},-\delta]$ for some $\delta>0$
\begin{defn}
Let $P_0:\Omega_{\infty}\rightarrow\mathbf{R}$ be a convex function, and $h_0\in L^2(\Omega_2)$. Then we say $(P_0,h_0)$ is a generalized data if the following holds:\\
$(i)$$\nabla P_0\in L^2(\Omega_{h_0})$ and $-\frac{1}{\delta}\leq\partial_{x_3}P_0\leq-\delta$\\
$(ii)$$h_0\geq0$ and $\int h_0dx_1dx_2=1$\\
$(iii)$$P_0(x_1,x_2,h_0(x_1,x_2))=\frac{1}{2}(x_1^2+x_2^2)\,\,\forall (x_1,x_2)\in\Omega_2$\\
\end{defn}
Next we show that the notion of relaxed Lagrangian solution defined here is consistent with weak Lagrangian solution when the measure $\alpha_t$ is induced by some physical flow map, except possibly the inverse map may not exist.
\begin{lem}
Let $(P,h,\alpha)$ be relaxed Lagrangian solution such that $P\in C([0,T];H^1(\Omega_{\infty}))$ with admissible initial data. Suppose there exists a Borel map $F:[0,T]\times\Omega_{h_0}\rightarrow\Omega_{\infty}$, which is weakly continuous in the sense that $$\int_{\Omega_{h_0}}F_t(x)\psi(x)dx\rightarrow\int_{\Omega_{h_0}}F_{t_0}(x)\psi(x)dx\,\,\,\forall\psi\in C_b(\mathbf{R}^3;\mathbf{R}^3)\,\,\,as\,\,t\rightarrow t_0\in[0,T]$$
such that $\alpha_t=(id\times F_t)_{\sharp}\sigma_{h_0}$. Then $(P,h,F)$ is a weak Lagrangian solution except possibly $(iv)$ of Definition 3.7.
\end{lem}
\begin{proof}
First we observe that the assumption $P\in L^{\infty}([0,T],W^{1,\infty}(\Omega_{\infty}))\bigcap C([0,T],H^1(\Omega_{\infty}))$ implies that $P\in C([0,T],W^{1,r}(\Omega_{\infty}))\,\,\,\forall r<\infty$. Also observe the following
$$\int_{\Omega_{h_0}}|F_t(x)|^2dx=\int_{\Omega_{h(t,\cdot)}}|\hat{x}|^2d\hat{x}\rightarrow\int_{\Omega_{h(t_0,\cdot)}}|\hat{x}|^2dx=\int_{\Omega_{h_0}}|F_{t_0}(x)|^2dx\,\,as\,\,t\rightarrow t_0$$
The equality above is because
$$\sigma_{h(t,\cdot)}=proj_{\hat{x}}\alpha_t=F_{t\sharp}\sigma_{h_0}$$
and the limit happens is because of the assumption $h\in C([0,T]\times\bar{\Omega}_2)$. Since $F_t$ is always bounded, this combines with weak continuity implies $F\in C([0,T],L^r(\sigma_{h_0}\mathcal{L}^3))$,$\forall r<\infty$

Now it only remains to check that for $Z(t,x)=\nabla P_t(F_t(x))$(well-defined thanks to the above observed $F_{t\sharp}\sigma_{h_0}=\sigma_{h(t,\cdot)}$), it satisfies $(v)$, but this is seen by taking $\xi(y)=y_k$ in (4.1).
\end{proof}
Here is the main result of this section.
\begin{thm}
Let $(P_0,h_0)$ be generalized data in the sense of Definition 4.4, then there exists a relaxed Lagrangian solution $(P,h,\alpha)$ having $(P_0.h_0)$ as initial data. Besides, we have the following continuity in time:

$(i)\,\,h\in C([0,T];L^2(\Omega_2))$

$(ii)\,\,\xi(P,\nabla P)\sigma_h\in C([0,T];L^1(\Omega_{\infty}))\,\,\forall\xi\in C_b(\mathbf{R}\times\mathbf{R}^3)$

\end{thm}

\subsection{Measure-valued solution in dual space }
Recall that when $\nu_0\in\mathcal{P}_{ac}(\mathbf{R}^3)$, the system in dual space can be written as
$$\partial_t\nu+\nabla\cdot(\nu\mathbf{w})=0$$
$$\mathbf{w}=J(y-\nabla P^*(t,y))\,\,\,y\in\mathbf{R}^3$$
$$(h,(id\times\nabla P)_{\sharp}\sigma_h)\,\,\,minimizes\,\,\,E_{\nu}(h,\gamma)$$
$$\nu|_{t=0}=\nu_0$$
See the end of section 2.\\
Since the second equation involves $\nabla P^*(t,y)$, it is not well-defined if $\nu$ is singular,we need to find an substitute for $\nabla P^*$. A natural candidate for this is to use barycentric projection.

\begin{defn}
Let $\mu,\nu\in\mathcal{P}(\mathbf{R}^3)$,$\lambda\in\Gamma(\mu,\nu)$, we say the Borel map (defined $\nu-a.e$ a,e ), $\bar{\gamma}:\mathbf{R}^3\rightarrow\mathbf{R}^3$ is the barycentric projection of $\lambda$ to $\nu$ if the following is true
$$\int_{\mathbf{R}^3}\xi(y)\cdot\bar{\gamma}(y)d\nu(y)=\int\int_{\mathbf{R}^3\times\mathbf{R}^3}\xi(y)\cdot xd\lambda(x,y)\,\,\,\forall\xi\in C_b(\mathbf{R}^3;\mathbf{R}^3)$$
\end{defn}
Now if $P_t$ be a family of convex functions and set $\nu_t=\nabla P_{t\sharp}\sigma_{h(t,\cdot)}$ , then we denote
$\bar{\gamma}_t(y)$ to be the barycentric projection of $(id\times\nabla P_t)_{\sharp}\sigma_h$ onto $\nu_t$, or equivalently
$$\int_{\mathbf{R}^3}\xi(y)\cdot\bar{\gamma}_t(y)d\nu_t(y)=\int_{\Omega_{\infty}}\xi(\nabla P_t(x))\cdot x\sigma_h(x)dx\,\,\,\forall\xi\in C_b(\mathbf{R}^3;\mathbf{R}^3)$$
It`s easy to see when $\nu_t<<\mathcal{L}^3$,then $\nabla P_t^*(y)=\bar{\gamma}_t(y)\,\,\,\nu_t-a.e$

Thus in the general case when $\nu_t$ is not necessarily absolutely continuous, a natural way to write the system in the dual space is(where $[0,T]\ni t\longmapsto\nu_t$ is a Borel family of measures)
\begin{equation}
\partial_t\nu+\nabla\cdot(\nu\mathbf{w})=0\,\,\,in\,\,\,[0,T]\times\mathbf{R}^3
\end{equation}
\begin{equation}
\mathbf{w}(t,y)=J(y-\bar{\gamma}_t(y))\,\,\,in\,\,\,[0,T]\times\mathbf{R}^3
\end{equation}
\begin{equation}
(h,(id\times\nabla P_t)_{\sharp}\sigma_h)\,\,\,minimizes\,\,\,E_{\nu}(h,\gamma)
\end{equation}
\begin{equation}
\nu|_{t=0}=\nu_0
\end{equation}
For immediate use, we first prove a lemma.
\begin{lem}
Let $\nu$ be a solution of above system, fix $t\in[0,T]$ and  $\gamma:=(id\times\nabla P_t)_{\sharp}\sigma_h$ be the optimal measure with marginals $\sigma_h$ and $\nu_t$ under quadratic cost, let $\bar{\gamma_t}$ be the barycenter projection defined in Definition 4.7, then
$$||J[id-\bar{\gamma}_t]||_{L^2(\nu_t)}\leq2( \max_{\Omega_2}|x|+W_2(\nu_t,\delta_0))$$
\end{lem}
\begin{proof}
Observe that
$$||J[id-\bar{\gamma}_t]||_{L^2(\nu_t)}=\sup_{\phi\in C_c(\mathbf{R}^3),||\phi||_{L^2}\leq1}\int J(y-\bar{\gamma}_t(y))\phi(y)d\nu_t(y)$$
But $$\int J(y-\bar{\gamma}_t(y))\phi(y)d\nu_t(y)=\int J(\nabla P_t(x)-x)\phi(\nabla P_t(x))d\sigma_{h(t,\cdot)}(x)$$$$\leq(\int(x_1-\partial_1P_t(x))^2+(x_2-\partial_2P_t(x))^2d\sigma_{h(t,\cdot)}(x))^{\frac{1}{2}}(\int\phi^2(\nabla P_t(x))d\sigma_{h(t,\cdot)})^{\frac{1}{2}}$$$$\leq[2(\max_{\Omega_2}|x|^2+\int_{\Omega_{h(t,\cdot)}}|\nabla P_t(x)|^2dx)]^{\frac{1}{2}}||\phi||_{L^2(\nu_t)}$$
But since $\nabla P_{t\sharp}\sigma_{h(t,\cdot)}=\nu_t$, we know
$$\int_{\Omega_{h(t,\cdot)}}|\nabla P_t(x)|^2dx=W_2^2(\nu_t,\delta_0)$$
\end{proof}

We denote by $AC^p(0,T;\mathcal{P}_2(\mathbf{R}^3))$(See also \cite{gradient-flows} chapter 8) to be the set of all paths $\mu:[0,T]\ni t\rightarrow \mu_t\in\mathcal{P}_2(\mathbf{R}^3)$ for which there exists $\beta\in L^p(0,T)$ such that
$$W_2(\mu_s,\mu_t)\leq\int_s^t\beta(\tau)d\tau\,\,\,\forall s,t\in[0,T]$$
First we want to study the stability of the measure valued dual space solution under perturbations of initial data.
\begin{prop}
Suppose $\{\nu_0^n,\nu_0\}_{n\geq1}\subset\mathcal{P}_2(\mathbf{R}^3) $ with $supp\,\,\nu_0^n,\nu_0\subset\mathbf{R}^2\times[-\frac{1}{\delta},-\delta]$ and such that
$$\nu_0^n\rightarrow\nu_0\,\,\,narrowly\,\,\,and\,\,\,\sup_{n\geq1}M_2(\nu_0^n)<\infty$$
Let $\nu^n\in AC^{\infty}(0,T;\mathcal{P}_2(\mathbf{R}^3))$ be solutions to the dual space system corresponding to initial data $\nu_0^n$, and also that $supp\,\,\nu_t^n\subset\mathbf{R}^2\times[-\frac{1}{\delta},-\delta]$. Then a subsequence $\nu^n$ converges to a solution $\nu\in AC^{\infty}(0,T;\mathcal{P}_2(\mathbf{R}^3))$ with initial data $\nu_0$,i.e

$(i)\,\,W_p(\nu_t^n,\nu_t)\rightarrow0$ for all $t\in[0,T]$, and all $p\in[1,2)$

$(ii)\,\,\nu\in AC^{\infty}(0,T;\mathcal{P}_2(\mathbf{R}^3))$ is a solution to dual space system with initial data $\nu_0$
\end{prop}
\begin{proof}
First we have for each n, and $0\leq s\leq t\leq T$
\begin{equation}
W_2(\nu_s^n,\nu_t^n)\leq\int_s^t||J[id-\bar{\gamma}_{\tau}^n||_{L^2(\nu_{\tau}^n;\mathbf{R}^3)}d\tau\leq2\int_s^tW_2(\nu_{\tau}^n,\delta_0)d\tau+2\max_{\Omega_2}|x|(t-s)
\end{equation}
The first inequality used \cite{gradient-flows} Theorem 8.3.1. The second inequality is by Lemma 4.12.

Now take $s=0$, we then have
$$W_2(\delta_0,\nu_t^n)\leq\sup_{n\geq1} W_2(\delta_0,\nu_0^n)+W_2(\nu_0^n,\nu_t^n)$$$$\leq\sup_{n\geq1}W_2(\delta_0,\nu_0^n)+2\int_0^tW_2(\nu_{\tau}^n,\delta_0)d\tau+2\max_{\Omega_2}|x|t$$
Then we can apply Gronwall to obtain
\begin{equation}
W_2(\delta_0,\nu_t^n)\leq\sup_{n\geq1}W_2(\delta_0,\nu_0^n)e^{2T}+(e^{2T}-1)T\max_{\Omega_2}|x|\,\,\,0\leq t\leq T
\end{equation}
which means $\{\nu_t^n\}_{t\in[0,T],n\geq1}$ are uniformly bounded in $\mathcal{P}_2(\mathbf{R}^3)$, so it is tight.\\
Combining (4.14),(4.15), we get
\begin{equation}
W_2(\nu_s^n,\nu_t^n)\leq2e^{2T}[\sup_{n\geq1}W_2(\delta_0,\nu_0^n)+(T+1)\max_{\Omega_2}|x|]|t-s|
\end{equation}
which means $\{\nu_t^n\}_{n\geq1}$ are equi-continuous in t under the distance $W_2$. This allows us to take a subsequence(not relabeled), such that
$$\nu_t^n\rightarrow\nu_t\,\,\,narrowly\,\,\,\forall t\in[0,T]$$
Since $\sup_{n\geq1}\int|y|^2d\nu_t^n<\infty$, we see
$$\int |y|^pd\nu_t^n\rightarrow\int|y|^pd\nu_t\,\,\,\forall p\in[1,2)$$
this implies $W_p(\nu_t^n,\nu_t)\rightarrow0$ by \cite{gradient-flows} Proposition 7.1.5. Also it`s easy to see $\nu_t\in AC^{\infty}(0,T;\mathcal{P}_2(\mathbf{R}^3))$ by (4.12). Finally let`s check $\nu_t$ solve the dual space system. Take $\phi(y,t)=\chi(t)\eta(y)$ be test function such that $\chi(T)\equiv0$, then for each n, we have
$$\int_0^T\int_{\mathbf{R}^3}\chi^{\prime}(t)\eta(y)d\nu_t^n(y)+\int_0^T\int_{\mathbf{R}^3}J(y-\bar{\gamma}_t^n(y))\cdot\nabla\eta(y)d\nu_t^n(y)$$
$$+\int_{\mathbf{R}^3}\chi(0)\eta(y)d\nu_0^n(y)=0$$
where $\bar{\gamma}_t^n(y)$ is the barycentric projection of $(id\times\nabla P_t^n)_{\sharp}\sigma_{h_t^n}$ onto $\nu_t^n$.

Because of narrow convergence, and $W_p$ convergence for $p\in[1,2)$ already noted,we see $\int_0^T\int_{\mathbf{R}^3}\chi^{\prime}(t)\eta(y)d\nu_t^n(y)$,$\int_{\mathbf{R}^3}\chi(0)\eta(y)d\nu_0^n(y)$,$\int_0^T\int_{\mathbf{R}^3}Jy\cdot\nabla\eta(y)d\nu_t^n(y)$ will converge to the right limit,while
$$\int_0^T\int_{\mathbf{R}^3}J\bar{\gamma}_t^n(y)\cdot\nabla\eta(y)d\nu_t^n(y)=\int_0^T\int_{\Omega_{\infty}}Jx\cdot
\nabla\eta(\nabla P^n_t(x))\sigma_{h_t^n}(x)dx$$
Notice that above integrand does not involve $x_3$, by the definition of $J$, and $x_1,x_2$ are bounded, we can conclude by using Theorem 2.53$(iii)$ that
$$\int_0^T\int_{\Omega_{\infty}}Jx\cdot\nabla \eta(\nabla P_t^n(x))\sigma_{h_t^n}(x)dx\rightarrow\int_0^T\int_{\Omega_{\infty}}Jx\cdot\nabla \eta(\nabla P_t(x))\sigma_{h_t}(x)dx$$$$=\int_0^T\int_{\mathbf{R}^3}J\bar{\gamma}_t(y)\cdot\nabla\eta(y)d\nu_t^n(y)$$
\end{proof}
As a corollary to above proposition, we can deduce the existence of measure-valued solution in dual space with possibly unbounded support.
\begin{cor}
Let $\nu_0\in\mathcal{P}_2(\mathbf{R}^3)$ and such that $supp\,\,\nu_0\subset\mathbf{R}^2\times[-\frac{1}{\delta},-\delta]$, then there exists a family of measures $[0,T]\ni t\longmapsto\nu_t$ such that $\nu\in AC^{\infty}(0,T;\mathcal{P}_2(\mathbf{R}^3))$,$supp\,\,\nu_t\subset\mathbf{R}^2\times[-\frac{1}{\delta},-\delta]$ and solves the dual space system.
\end{cor}
\begin{proof}
Define $\nu_0^n=\frac{1}{\nu_0(B(0,n))}[\nu_0\chi_{B(0,n)}]*j_{\frac{1}{n}}$, then $\nu_0^n\in\mathcal{P}_2(\mathbf{R}_3)$, $supp\,\,\nu_0^n\subset\mathbf{R}^2\times[-\frac{1}{\delta},-\delta]$ and compact. Also we have $\nu_0^n\rightarrow\nu_0$ narrowly. Then we can use theorem 2.30 to get a weak solution $\nu_t^n$ in dual space solution. The argument used to show (4.12) shows $\nu_t^n\in AC^{\infty}(0,T;\mathcal{P}_2(\mathbf{R}^3))$ and has $\nu_0^n$ as initial data with $supp\,\,\nu_t^n\subset\mathbf{R}^2\times[-\frac{1}{\delta},-\delta]$. The previous proposition gives us a solution $\nu_t$ with initial data $\nu_0$\\
\end{proof}
\subsection{Existence of relaxed Lagrangian solution with generalized data}
In this section, we will prove the existence of relaxed solutions with generalized initial data.
First we show that for the generalized data $(P_0,h_0)$ defined previously, $P_0|_{\Omega_{h_0}}$ gives a maximizer. This is a generalization of Proposition 3.3. More precisely, we have
\begin{lem}
Let $(P_0,h_0)$ be a generalized data, put $\nu_0=\nabla P_{0\sharp}\sigma_{h_0}$. Let $(\tilde{P}_0,\tilde{R}_0)$ be a pair of maximizer of $J_{\nu_0}(P,R)$, which satisfies (3.39),(3.40) , then $P_0=\tilde{P}_0$ on $\Omega_{h_0}$ and $(h_0,(id\times\nabla P_0)_{\sharp}\sigma_{h_0})$ is the minimizer of $E_{\nu_0}(h,\gamma)$
\end{lem}
\begin{proof}
The argument is similar to Proposition 3.3. The difference is that one need to define
$$\hat{P}_0(x)=\max(P_0(x),\frac{1}{2}(x_1^2+x_2^2))$$
and
$$\hat{R}_0(y)=\sup_{x\in\Omega_{\infty}}(x\cdot y-\hat{P}_0(x))$$
This definition ensures $\hat{R}_0$ is finite. The rest of the argument works in the same way as Proposition 3.3.
\end{proof}
\begin{rem}
Recall Definition 4.2, the same argument as in Proposition 3.3 shows that for each fixed $t\in[0,T]$, if we put $\nu_t=\nabla P_{t\sharp}\sigma_{h(t,\cdot)}$, then $(h_t,(id\times\nabla P_t)_{\sharp}\sigma_{h_t})$ is the minimizer of $E_{\nu_t}(h,\gamma)$, $P_t$ is the restriction of some maximizer of $J_{\nu_t}(P,R)$ restrcted on $\Omega_{h_t}$
\end{rem}

First we observe that, similar to the fixed boundary case, the relaxed Lagrangian solution will give rise to a measure valued solution in dual space.
\begin{prop}
Let $(P,h,\alpha)$ be a relaxed Lagrangian solution and let $\nu_t=\nabla P_{t\sharp}\sigma_{h(t,\cdot)}$, then

$(i)$$\nu$ solves the dual space system with initial data $\nu_0=\nabla P_{0\sharp}\sigma_{h_0}$,

$(ii)$$\nu_t\in AC^{\infty}(0,T;\mathcal{P}_2(\mathbf{R}^3))$
\end{prop}
\begin{proof}
We choose test function as $\phi(t,y)=\chi(t)\psi(y)$, then we can compute
$$\int_0^T\int_{\mathbf{R}^3}\chi^{\prime}(t)\psi(y)d\nu_t(y)dt=\int_0^T\int_{\Omega_{\infty}}\chi^{\prime}(t)\psi(\nabla P_t(x))d\sigma_{h(t,\cdot)}(x)dt=\int_0^T\int_{\Omega_{\infty}^2}\psi(\nabla P_t(\hat{x}))\chi^{\prime}(t)d\alpha_t(x,\hat{x})$$$$=-\int_0^T\int_{\Omega_{\infty}^2}\nabla\psi(\nabla P_t(\hat{x}))\cdot J(\nabla P_t(\hat{x})-\hat{x})\chi(t)d\alpha_t(x,\hat{x})dt-\int_{\Omega_{h_0}}\psi(\nabla P_0(x))\chi(0)dx$$
In the above, we have used (4.1) and the marginal properties of $d\alpha_t(x,\hat{x})$.
Now $$\int_0^T\int_{\Omega_{\infty}^2}\nabla\psi(\nabla P_t(\hat{x}))\cdot J(\nabla P_t(\hat{x}))\chi(t)d\alpha_t(x,\hat{x})=\int_0^T\int_{\Omega_{h(t,\cdot)}}\nabla\psi(\nabla P_t(\hat{x}))\cdot J(\nabla P_t(\hat{x}))\chi(t)dxdt$$$$=\int_0^T\int_{\mathbf{R}^3}\nabla\psi(y)\cdot
Jy\chi(t)d\nu_t(y)dt$$
while the term
$$\int_0^T\int_{\Omega_{\infty}^2}\nabla\psi(\nabla P_t(\hat{x}))\cdot J\hat{x}\chi(t)d\alpha_t(x,\hat{x})dt
=\int_0^T\int_{\Omega_{\infty}}\nabla\psi(\nabla P_t(x))\cdot Jx\chi(t)d\sigma_{h(t,\cdot)}(x)dt$$
$$=\int_0^T\int_{\mathbf{R}^3}\nabla\psi(y)\cdot J\bar{\gamma}_t(y)\chi(t)d\nu_t(y)dt$$
The last term
$$\int_{\Omega_{h_0}}\psi(\nabla P_0(x))\chi(0)dx=\int_{\mathbf{R}^3}\psi(y)\chi(0)d\nu_0(y)$$
So we proved $(i)$.
To see $(ii)$ is true, we notice
$$W_2(\nu_s,\nu_t)\leq\int_s^t||J(id-\bar{\gamma}_{\tau})||_{L^2(\nu_{\tau})}d\tau\leq\int_s^t2(\max_{\Omega_2}|x|+W_2(\nu_{\tau},\delta_0))d\tau$$
In the above, we used Lemma 4.12. Now we can use Gronwall inequality to conclude that $\{\nu_t\}_{t\in[0,T]}$ is bounded in $\mathcal{P}_2(\mathbf{R}^3)$.
\end{proof}
The existence result of relaxed solution is an easy consequence of the following lemma.
\begin{lem}
Let $(P_0^n,h_0^n)$,$(P_0,h_0)$ be generalized initial data in the sense of Definition 4.4, Suppose also that
$$\sup_{n\geq1}||\nabla P_0^n||_{L^2(\Omega_{h_0^n})}<\infty$$
and
$$h_0^n\rightarrow h_0\,\,in\,\,L^1(\Omega_2)\,\,\,\,\xi(P_0^n,\nabla P_0^n)\sigma_{h_0^n}\rightarrow\xi(P_0,\nabla P_0)\sigma_{h_0}\,\,\,in\,\,L^1(\Omega_{\infty})$$
$\forall \xi\in C_b(\mathbf{R}\times\mathbf{R}^3)$

Let $(P^n,h^n,\alpha^n)$ be relaxed Lagrangian solution with initial data $(P_0^n,h_0^n)$, then there exists a subsequence
which converges to a relaxed solution with initial data $(P_0,h_0)$. More precisely, for any $t\in[0,T]$

$(i)\,\,h^n_t\rightarrow h_t\,\,\,in\,\,L^r(\Omega_2)\,\,\forall r\in[1,2)$;

$(ii)\,\,\alpha^n\rightarrow\alpha$ narrowly;

$(iii)\,\,\xi(P^n_t,\nabla P_t^n)\sigma_{h^n}\rightarrow\xi(P_t,\nabla P_t)\sigma_h$ in $L^1(\Omega_{\infty})\,\,\forall\xi\in C_b(\mathbf{R}\times\mathbf{R}^3)$
\end{lem}
\begin{proof}
Define $\nu_t^n=\nabla P^n_{t\sharp}\sigma_{h_t^n}$, then by Proposition 4.20, we know $\nu^n$ solves the dual space system with $\nu^n\in AC^{\infty}(0,T;\mathcal{P}_2(\mathbf{R}^3))$. Since $-\frac{1}{\delta}\leq\partial_{x_3}P_t\leq-\delta$, we also know that $supp\,\,\nu_t^n\subset\mathbf{R}^2\times[-\frac{1}{\delta},-\delta]$

 Define $\nu_0=\nabla P_{0\sharp}\sigma_{h_0}$.
By assumption, one has $supp\,\,\nu_0^n,\nu_0\subset\mathbf{R}^2\times[-\frac{1}{\delta},-\delta]$, Also
$$\int_{\mathbf{R}^3}\eta(y)d\nu_0^n(y)=\int_{\Omega_{\infty}}\eta(\nabla P_0^n(x))\sigma_{h_0^n}(x)dx\rightarrow\int_{\Omega_{\infty}}\eta(\nabla P_0(x))\sigma_{h_0}(x)dx$$$$=\int_{\mathbf{R}^3}\eta(y)d\nu_0(y)\,\,\,\forall\eta\in C_b(\mathbf{R}^3)$$
So $\nu_0^n\rightarrow\nu_0$ narrowly. Also we have $\sup_{n\geq1}M_2(\nu_0^n)<\infty$, because we assumed $\sup_{n\geq1}||\nabla P_0^n||_{L^2(\Omega_{h_0^n})}<\infty$.

We are now in a position to apply Proposition 4.13 to get a subsequence of $\nu_t^n$ which converges to a solution $\nu\in AC^{\infty}(0,T;\mathcal{P}_2(\mathbf{R}^3))$ which solves the dual space system with initial data $\nu_0$.\\
By (4.11) in the proof of Proposition 4.13, we know $\{\nu_t^n\}_{t\in[0,T],n\geq1}$ is uniformly bounded in $\mathcal{P}_2(\mathbf{R}^3)$, hence $\{||\nabla P_t^n||_{L^2(\Omega_{h_t^n})}\}_{t\in[0,T],n\geq1}$ are uniformly bounded. Remark 4.19 shows for each $t\in[0,T]$,$(h_t^n,(id\times\nabla P_t^n)_{\sharp}\sigma_{h_t^n})$ is the unique minimizer of $E_{\nu_t^n}(h,\gamma)$. Now let $(\tilde{P}_t^n,\tilde{R}_t^n)$,$(\tilde{P},\tilde{R})$ be  maximizers of $J_{\nu_t^n}(P,R)$,$J_{\nu_t}(P,R)$ respectively satisfying (2.39),(2.40), then by Theorem 2.53, we conclude for each $t\in[0,T]$
$$h_t^n\rightarrow h_t\,\,in\,\,L^r(\Omega_2)\,\,\,for\,\,\,any\,\,\,r\in[1,2)$$
$$\xi(\tilde{P}_t^n,\nabla\tilde{P}_t^n)\sigma_{h_t^n}\rightarrow\xi(\tilde{P}_t
.
,\nabla\tilde{P}_t)\,\,in\,\,L^1(\Omega_{\infty})\,\,\,\forall \xi\in C_b(\mathbf{R}\times\mathbf{R}^3)$$

Also we know that $\{||h_t^n||_{L^2(\Omega_2)}\}_{t\in[0,T],n\geq1}$ is bounded by a universal constant by Lemma 2.46.
By Remark 4.19, we actually have $\tilde{P}_t^n=P_t^n$ on $\Omega_{h_t^n}$, so we really have
$$\xi(P_t^n,\nabla P_t^n)\sigma_{h_t^n}\rightarrow\xi({P}_t,\nabla{P}_t)\sigma_{h_t}\,\,in\,\,L^1(\Omega_{\infty})\,\,\,\forall \xi\in C_b(\mathbf{R}\times\mathbf{R}^3)$$
Finally since for each n,  $proj_x\alpha^n_t=\sigma_{h_0^n},proj_{\hat{x}}\alpha_t^n=\sigma_{h_t^n}$. Because of the $L^2$ bound of $h_t^n$, we know $\{\alpha^n\}$, as a measure on $[0,T]\times\Omega_{\infty}\times\Omega_{\infty}$, is tight. Therefore up to a subsequence, we can assume $\alpha^n\rightarrow\alpha$ narrowly, and the limit $\alpha$ must disintegrate as $d\alpha(x,\hat{x})=d\alpha_tdt$, with
$proj_x\alpha_t=\sigma_{h_0}$ and $proj_{\hat{x}}\alpha_t=\sigma_{h_t}$. So far we have checked point (i)-(iii) in Definition 4.2. (v) follows from the assumed convergence of the initial data and the fact that $P_0^n\rightarrow P_0$ locally uniformly on $\Omega\times\{0\}$.  It only remains to check $(\tilde{P},h,\alpha)$ satisfies point (iv). For each fixed n, we have
\begin{equation}\int_0^T\int_{\Omega_{\infty}^2}\xi(\nabla P^n_t(\hat{x}))\partial_t\psi(t,x)d\alpha^n_t(x,\hat{x})dt+\int_0^T\int_{\Omega_{\infty}^2}\nabla \xi(\nabla P^n_t(\hat{x}))\cdot J(\nabla P^n_t(\hat{x})-\hat{x})\psi(t,x)d\alpha^n_t(x,\hat{x})dt$$$$+\int_{\Omega_{h_0^n}}\xi(\nabla P^n_0(x))\psi(0,x)dx=0\,\,\,\forall \xi\in C_c^1(\mathbf{R}^3)
\end{equation}
We wish to pass each term to the limit. First it`s obvious that
$$\int_{\Omega_{h_0^n}}\xi(\nabla P^n_0(x))\psi(0,x)dx
\rightarrow \int_{\Omega_{h_0}}\xi(\nabla P_0(x))\psi(0,x)dx$$
since this convergence happens $\mathcal{L}^3-$a.e on $\Omega_{h_0}$ and $h_0^n\rightarrow h_0$ uniformly as noted above.
Next we look at first term. Fix $\epsilon>0$, choose $K>0$, so that $\frac{\sup_{n,t}||h_t^n||_{L^2(\Omega_2)}}{K}\leq \epsilon$. This is possible because of the universal $L^2$ bound on $h$, see Lemma 2.46.  We choose $g\in C_b(\mathbf{R}^3,\mathbf{R}^3)$, such that
$$\mathcal{L}^4(\{(t,\hat{x})\in[0,T]\times\Omega_K|\nabla P_t(\hat{x})\neq g(t,\hat{x})\})<\epsilon$$
Then we can write
$$\int_0^T\int_{\Omega_{\infty}^2}\xi(\nabla P_t^n(\hat{x}))\partial_t\psi(t,x)d\alpha^n(t,x,\hat{x})-\int_0^T\int_{\Omega_{\infty}^2}\xi(\nabla P_t(\hat{x}))\partial_t\psi(t,x)d\alpha(t,x,\hat{x})$$
\begin{equation}
=\int_0^T\int_{\Omega_{\infty}^2}(\xi(\nabla P_t^n(\hat{x}))-\xi(\nabla P_t(\hat{x})))\partial_t\psi(t,x)d\alpha^n(t,x,\hat{x})
\end{equation}
$$+\int_0^T\int_{\Omega_{\infty}^2}(\xi(\nabla P_t(\hat{x}))-\xi(g(t,\hat{x})))\partial_t\psi(t,x)d\alpha^n(t,x,\hat{x})$$
$$+\int_0^T\int_{\Omega_{\infty}^2}\xi(g(t,\hat{x}))\partial_t\psi(t,x)d(\alpha^n(t,x,\hat{x})-\alpha(t,x,\hat{x}))$$
$$+\int_0^T\int_{\Omega_{\infty}^2}(\xi(g(t,\hat{x}))-\xi(\nabla P_t(\hat{x})))\partial_t\psi(t,x)d\alpha(t,x,\hat{x})=A+B+C+D$$
First we estimate $B$. We notice that
\begin{equation}
\int_0^T\int_{\Omega_{\infty}^2}(\xi(\nabla P_t(\hat{x}))-\xi(g(t,\hat{x})))\partial_t\psi(t,x)d\alpha^n(t,x,\hat{x})$$$$\leq||\partial_t\Psi||_{L^{\infty}}\int_0^T\int_{\Omega_{\infty}}|\xi(\nabla P_t(\hat{x}))-\xi(g(t,\hat{x}))|\sigma_{h_t^n}(\hat{x})
\end{equation}
$$\leq||\partial_t\psi||_{L^{\infty}}[\int_0^T\int_{\Omega_K}|\xi(\nabla P_t(\hat{x}))-\xi(g(t,\hat{x}))|\sigma_{h_t^n}(\hat{x})+||\xi||_{L^{\infty}}\int_0^T\int_{\Omega_2}(h_t^n-K)^+dx_1dx_2]$$
$$\leq||\partial_t\psi||_{L^{\infty}}[||\xi||_{L^{\infty}}\epsilon+||\xi||_{L^{\infty}}\frac{||h^n_t||_{L^2(\Omega_2)}}{K}]\leq2||\partial_t\psi||_{L^{\infty}}||\xi||_{L^{\infty}}\epsilon$$
The same estimate also works for $D$.
while $C\rightarrow0$ by narrow convergence.\\
It only remains to estimate $A$. $$\int_0^T\int_{\Omega_{\infty}^2}(\xi(\nabla P_t^n(\hat{x}))-\xi(\nabla P_t(\hat{x})))\partial_t\psi(t,x)d\alpha^n(t,x,\hat{x})\leq$$$$||\partial_t\psi||_{L^{\infty}}
\int_0^T\int_{\Omega_{\infty}}|\xi(\nabla P_t^n(\hat{x}))-\xi(\nabla\tilde{P}_t(\hat{x}))|\sigma_{h_t^n}(\hat{x})d\hat{x}dt$$
Then we can write
$$|\xi(\nabla P_t^n(\hat{x}))-\xi(\nabla\tilde{P}_t(\hat{x}))|\sigma_{h_t^n}(x)\leq|\xi(\nabla P_t^n(\hat{x}))\sigma_{h_t^n}(\hat{x})-\xi(\nabla\tilde{P}_t(\hat{x}))\sigma_{h_t}(\hat{x})|+$$$$||\xi||_{L^{\infty}}|\sigma_{h_t^n}(\hat{x})-\sigma_{h_t}(\hat{x})|$$
By the convergence already noted, their integral goes to zero.  So we can pass the first term of (4.22) to limit.

For the second term, the term $\int_0^T\int_{\Omega_{\infty}^2}\nabla\xi(\nabla P_t^n(\hat{x}))\cdot J\hat{x}\psi(t,x)d\alpha_t^n(x,\hat{x})dt$ can be passed to limt in the same way as first term, if one notices that $J\hat{x}$ does not involve $\hat{x}_3$. We only need to deal with $\int_0^T\int_{\Omega_{\infty}^2}\nabla\xi(\nabla P_t^n(\hat{x}))\cdot J\nabla P_t^n(\hat{x})\psi(t,x)d\alpha_t^n(x,\hat{x})dt$ Let $H\in C_b(\mathbf{R}^3)$ be a cut-off function such that $H(y)=1$ if $|y|\leq K$, $H(y)=0$ if $|y|\geq K+1$, $K$ to be determined below. Then we can write
\begin{equation}
\int_0^T\int_{\Omega_{\infty}^2}\nabla\xi(\nabla P_t^n(\hat{x}))\cdot J\nabla P_t^n(\hat{x})\partial_t\psi d\alpha_t^n(\hat{x},x)-\int_0^T\int_{\Omega_{\infty}^2}\nabla\xi(\nabla P_t(\hat{x}))\cdot J\nabla P_t(\hat{x})\partial_t\psi d\alpha_t(\hat{x},x)
\end{equation}
$$\leq ||\partial_t\psi||_{L^{\infty}}||\nabla\xi||_{L^{\infty}}[\int_0^T\int_{\Omega_{\infty}}|\nabla P_t^n(\hat{x})(1-H(\nabla P_t^n(\hat{x}))|\sigma_{h_t^n}(\hat{x})d\hat{x}$$$$+\int_0^T\int_{\Omega_{\infty}}|\nabla P_t(\hat{x})(1-H(\nabla P_t(\hat{x}))|\sigma_{h_t}(\hat{x})d\hat{x}]$$$$+[\int_0^T\int_{\Omega_{\infty}^2}\nabla\xi(\nabla P_t^n(\hat{x}))\cdot J\nabla P_t^n(\hat{x})H(\nabla P_t^n(\hat{x}))\partial_t\psi d\alpha_t^n(\hat{x},x)$$$$-\int_0^T\int_{\Omega_{\infty}^2}\nabla\xi(\nabla P_t(\hat{x}))\cdot J\nabla P_t(\hat{x})H(\nabla P_t(\hat{x}))\partial_t\psi d\alpha_t(\hat{x},x)]:=A+B+C$$
For $A$, since $\nu_t^n=\nabla P_t^n\sharp\sigma_{h_t^n}$, we can choose $K$ large, so that
$$A=\int_0^T\int_{\mathbf{R}^3}|y(1-H(y)|d\nu_t^n(y)\leq\int_0^T\int_{\{|y|\geq K\}}|y|d\nu_t^ndt\leq\frac{T\sup_{n,t}M_2(\nu_t^n)}{K}\leq\epsilon$$
The term $B$ is exactly the same as above. For term $C$, it can be dealt with in the same way as first term of (4.22) already shown above, since the integrands in $C$ is bounded because of the cut-off.
\end{proof}
Now we are ready to prove Theorem 4.6, the existence of relaxed Lagrangian solutions.\\
\begin{proof}
Let $\nu_0=\nabla P_{0\sharp}\sigma_{h_0}$, then $\nu_0\in\mathcal{P}_2(\mathbf{R}^3)$ and $supp\,\,\nu_0\subset\mathbf{R}^2\times[-\frac{1}{\delta},-\delta]$. Define
$\nu_0^n=\frac{1}{\nu_0(B(0,n))}[\chi_{B(0,n)}\nu_0]*j_{\frac{1}{n}}$. Then $\nu_0^n$ has compact support and smooth and $\nu_0^n\rightarrow\nu_0$ narrowly. Let $(P_0^n,R_0^n)$ be the maximizer of $J^{H_n}_{\nu_0^n}(P,R)$ with $H_n$ taken large enough. $(h_0^n,\gamma_0^n)$ be the minimizer of $E_{\nu_0^n}(h,\gamma)$, then $(P_0^n,h_0^n)$  are admissible initial data. We know by Lemma 2.54 that it is also the maximizer of $J_{\nu_0^n}(P,R)$ after suitable extension of $P_0^n$, hence they are also generalized data in the sense of Definition 4.2. Also by Theorem 2.53, we have $h_0^n\rightarrow h_0$ in $L^r(\Omega_2)$, for any $r\in[1,2)$, and $\xi(P^n,\nabla P_0^n)\sigma_{h_0^n}\rightarrow\xi(P_0,\nabla P_0)\sigma_{h_0}$ in $L^1(\Omega_{\infty})$. Let $(P^n,h^n,\alpha^n)$ be the Lagrangian solution with initial data $(P_0^n,h_0^n)$ given by Theorem 3.9. Then previous lemma gives us a relaxed solution with initial data $(P_0,h_0)$\\
The continuity property in time of $h$ is given by Theorem 2.53(iv) upon noticing that $\nu\in AC^{\infty}(0,T;\mathcal{P}_2(\mathbf{R}^3))$
\end{proof}

\section{Appendix}
Here we prove that
\begin{lem}
Given $(P_0,R_0)$ be a pair of convex conjugate maximizer of $J_{\nu}^H(P,R)$, let $h_0=h_{P_0}^H$ be such that $0\leq h_0<H$, Suppose also $\Omega_2\subset B_D(0) $, where $\Lambda=B_D(0)\times[-\frac{1}{\delta},-\delta]$, define
$$R_1(y)=\sup_{x\in\Omega_H}[x\cdot y-\max(P_0(x),\frac{1}{2}(x_1^2+x_2^2))]$$
$$P_1(x)=\sup_{y\in\Lambda}(x\cdot y-R_1(y))$$
Then the following holds

$(i)$$(P_1,R_1)$ are convex conjugate over $\Omega_H,\Lambda$

$(ii)$$P_1(x_1,x_2,0)\geq\frac{1}{2}(x_1^2+x_2^2)\,\,\,\forall (x_1,x_2)\in\Omega_2$

$(iii)$$P_1(x_1,x_2,0)\leq\frac{1}{2}(x_1^2+x_2^2)\,\,\,whenever\,\,\,h_0(x_1,x_2)=0$

$(iv)$$(P_1,R_1)$ is a maximizer of $J_{\nu}^H(P,R)$ and $h_{P_1}^H=h_0$

\end{lem}
\begin{proof}
First we prove $(i)$, denote $\hat{P}_0(x)=\max(P_0(x),\frac{1}{2}(x_1^2+x_2^2))$, we need to show
$$R_1(y)=\sup_{x\in\Omega_H}(x\cdot y-P_1(x))$$
By definition of $R_1(y)$, we know
$$R_1(y)+\hat{P}_0(x)\geq x\cdot y\,\,\,\forall x\in\Omega_H\,\,,\forall y\in\Lambda$$Hence
$$\hat{P}_0(x)\geq P_1(x)\,\,\,x\in\Omega_H$$
So
$$R_1(y)\leq\sup_{x\in\Omega_H}(x\cdot y-P_1(x))$$
By the definition of $P_1$, we have
$$P_1(x)+R_1(y)\geq x\cdot y\,\,\,\forall x\in\Omega_H\,\,,\forall y\in\Lambda$$
So
$$R_1(y)\geq\sup_{x\in\Omega_H}(x\cdot y-P_1(x))$$
$(i)$ is proved. Now we prove $(ii)$. We can observe that
$$R_1(y)\leq\sup_{x\in\Omega_H}(x\cdot y-\frac{1}{2}(x_1^2+x_2^2))$$
Fix $(x_1^0,x_2^0)\in\Omega_2$, by assumption, we can find $z_0\in[-\frac{1}{\delta},-\delta]$, such that $(x_1^0,x_2^0,z_0)\in\Lambda$, then
$$R_1(x_1^0,x_2^0,z_0)\leq\sup_{x\in\Omega_H}(x\cdot(x_1^0,x_2^0,z_0)-\frac{1}{2}(x_1^2+x_2^2))$$
$$=\sup_{\Omega_2}(x_1x_1^0+x_2x_2^0-\frac{1}{2}(x_1^2+x_2^2))\leq\frac{1}{2}[(x_1^0)^2+(x_2^0)^2]$$
In the equality above, we noticed $x_3z_0\leq0$. Therefore,
$$P_1(x_1^0,x_2^0,0)=\sup_{y\in\Lambda}(x_1^0y_1+x_2^0y_2-R_1(y))\geq(x_1^0)^2+(x_2^0)^2-R_1(x_1^0,x_2^0,z_0)\geq\frac{1}{2}[(x_1^0)^2+(x_2^0)^2]$$
$(ii)$ is proved. Now we prove $(iii)$. Fix $(x_1^0,x_2^0)\in\Omega_2$ such that $h_0(x_1^0,x_2^0)=0$, then
$(x_1^0,x_2^0,0)\notin\Omega_{h_0}$\\
But
$$R_1(y)=\max[\sup_{x\in\Omega_{h_0}}(x\cdot y-\hat{P}_0(x)),\sup_{x\notin\Omega_{h_0}}(x\cdot y-\hat{P}_0(x))]$$
On $\Omega_{h_0}$,$P_0(x)\geq\frac{1}{2}(x_1^2+x_2^2)$, so $\hat{P}_0(x)=P_0(x)$, otherwise $\hat{P}_0(x)=\frac{1}{2}(x_1^2+x_2^2)$\\
Hence
$$R_1(y)\geq\sup_{x\notin\Omega_{h_0}}(x\cdot y-\frac{1}{2}(x_1^2+x_2^2))$$
Then
$$(x_1^0,x_2^0,0)\cdot y-R_1(y)\leq \frac{1}{2}((x_1^0)^2+(x_2^0)^2)\,\,\,y\in\Lambda$$
Taking supremum over $y$, we obtain
$$P_1(x_1^0,x_2^0,0)\leq \frac{1}{2}((x_1^0)^2+(x_2^0)^2)$$
$(iii)$ is proved. Finally we prove $(iv)$. We start by showing that $P_1=P_0$ on $\Omega_{h_0}$
Indeed
$$R_1(y)=\sup_{x\in\Omega_H}(x\cdot y-\hat{P}_0(x))\leq\sup_{x\in\Omega_H}(x\cdot y-P_0(x))=R_0(y)$$
So
$$P_1(x)=\sup_{y\in\Lambda}(x\cdot y-R_1(y))\geq\sup_{y\in\Lambda}(x\cdot y-R_0(y))=P_0(x)$$
On the other hand
$$R_1(y)\geq\sup_{x\in\Omega_{h_0}}(x\cdot y-P_0(x))$$
So if $x\in\Omega_{h_0}$
$$P_1(x)=\sup_{y\in\Lambda}(x\cdot y-R_1(y))\leq P_0(x)$$
Combined with $(ii),(iii)$ above and recall the definition of $h_P^H$, we can deduce that $h_{P_0}^H=h_{P_1}^H=h_0$. Hence if we define $\gamma_0=(id\times\nabla P_0)_{\sharp}\sigma_{h_0}$, then $(h_0,\gamma_0)$ is the minimizer of $E_{\nu}(h,\gamma)$, So
$$J_{\nu}^H(P_1,R_1)=\int_{\Lambda}[\frac{1}{2}(y_1^2+y_2^2)-R_1(y)]d\nu(y)+\int_{\Omega_{h_0}}[\frac{1}{2}(x_1^2+x_2^2)-P_1(x)]dx$$$$=\int_{\Omega_H\times\Lambda}[\frac{1}{2}(y_1^2+y_2^2)+\frac{1}{2}(x_1^2+x_2^2)-R_1(y)-P_1(x)]d\gamma_0=E_{\nu}(h_0,\gamma_0)$$
Since $(P_1,R_1)$ convex conjugate, and $P_0=P_1$ on $\Omega_{h_0}$, $R_1(y)+P_1(x)=x\cdot y\,\,\,\gamma_0-a.e$
\end{proof}
\begin{rem}
We notice that above argument still works  even if $H=\infty$
\end{rem}

\begin{lem}
Suppose $(P_1,R_1)$ are maximizers of $J_{\nu}^H(P,R)$, convex conjugate over $\Omega_H$ and $\Lambda$. Suppose also that $P_1(x_1,x_2,0)=\frac{1}{2}(x_1^2+x_2^2)$ whenever $h_1(x_1,x_2)=0$, where $h_1=h_{P_1}^H<H$. We  define
 $$R_2(y)=\sup_{x\in\Omega_{h_1}\bigcup\{x_3=0\}}(x\cdot y-P_1(x))$$
and $$P_2(x)=\sup_{y\in\Lambda}(x\cdot y-R_2(y))$$
then

$(i)$ $P_1=P_2$ on $\Omega_{h_1}\bigcup\{x_3=0\}$

$(ii)$ $(P_2,R_2)$ are convex conjugate over $\Omega_H$ and $\Lambda$.

$(iii)$$(P_2,R_2) $ is also a maximizers and $h_{P_2}^H=h_1$.
\end{lem}
\begin{proof}
Take $x\in\Omega_{h_1}\bigcup\{x_3=0\}$, then $\forall y\in\Lambda$, we have
$$x\cdot y-R_2(y)=x\cdot y-\sup_{\bar{x}\in\Omega_{h_1}\bigcup\{x_3=0\}}(\bar{x}\cdot y-P_1(\bar{x}))\leq P_1(x)$$
Take supremum over y to get $$P_2(x)\leq P_1(x)\,\,\,(\forall x\in\Omega_{h_1}\bigcup\{x_3=0\})$$
On the other hand, by definition $R_2(y)\leq R_1(y)$, since $(P_1,R_1)$ conjugate over $\Omega_H$ and $\Lambda$, we obtain $$P_1(x)\leq P_2(x)\,\,\,(\forall x\in\Omega_H)$$$(i)$ is proved.\\
To see $(P_2,R_2)$ also convex conjuate over $\Omega_H$ and $\Lambda$, we only need to show
$$R_2(y)=\sup_{x\in\Omega_H}(x\cdot y-P_2(x))$$
Obviously $LHS\leq RHS$ by $(i)$.

On the other hand, for all $x\in\Omega_H$
$$x\cdot y-P_2(x)=x\cdot y-\sup_{\bar{y}\in\Lambda}(x\cdot\bar{y}-R_2(\bar{y}))\leq R_2(y)$$
So $(ii)$ is proved.\\
Finally we only need to see $J^H_{\nu}(P_1,R_1)=J^H_{\nu}(P_2,R_2)$. By assumption, we know $(h_1,(id\times\nabla P_1)_{\sharp}\sigma_{h_1})$ is the minimizer of $E_{\nu}(h,\gamma)$\\
By $(i)$, we know that
$$P_2(x_1,x_2,0)=\frac{1}{2}(x_1^2+x_2^2)\,\,\,whenever\,\,\,h_1(x_1,x_2)=0$$
Combining the fact that $P_1=P_2$ on $\Omega_{h_1}$, it`s easy to see $h_{P_2}^H=h_1$. The same argument as previous lemma $(iv)$ shows that $(P_2,R_2)$ is a maximizer.
\end{proof}
We derive the following corollary as an easy consequence of previous two lemmas.

\begin{cor}
Let $\Lambda=B_D(0)\times[-\frac{1}{\delta},-\delta]$ be such that $\Omega_2\subset B_D(0)$, and H is chosen such that $J_{\nu}^H(P,R)$ has convex conjugate maximizers over $\Omega_H$ and $\Lambda$, say $(P_0,R_0)$, and $h_0:=h_{P_0}^H<H$. Then there exists a maximizer $(P_2,R_2)$ of $J_{\nu}^H(P,R)$ which satisfies the following conditions:

$(i)$$(P_2,R_2)$ are convex conjugate over both $\Omega_{h_0}\bigcup\{x_3=0\},\Lambda$ and $\Omega_H,\Lambda$.

$(ii)$$P_2(x_1,x_2,0)=\frac{1}{2}(x_1^2+x_2^2)$ whenever $h_0(x_1,x_2)=0$.
\end{cor}
\begin{proof}
Let $(P_1,R_1)$ be the pair given by Lemma 5.1. The conclusion of Lemma 5.1 shows $(P_1,R_1)$ satisfies the assumptions of Lemma 5.2. Let $(P_2,R_2)$ be the pair given by  Lemma 5.2. Then such a pair is a maximizer by Lemma 5.2 $(iii)$. They are convex conjugate over $\Omega_{h_0}\bigcup\{x_3=0\}$ by their very definition. They are convex conjugate over $\Omega_H,\Lambda$ by Lemma 5.2 $(ii)$. $P_2$ satisfies $(ii)$ because of Lemma 5.2$(i)$ and Lemma 5.1 $(ii),(iii)$.
\end{proof}
The following lemma can be found in \cite{Evans} section 5.3 Theorem 1, and so we omit the proof.
\begin{lem}
Let $\Omega$ be a convex domain in $\mathbf{R}^d$, let $P:\Omega\rightarrow\mathbf{R}$ be a convex function such that $||P||_{L^1(\Omega)}<\infty$. Let $\Omega_1=\{x\in\Omega|dist(x,\Omega^c)>r\}$. Then there exists a constant $C=C(||P||_{L^1(\Omega)},r,d)$,such that
\begin{equation}
||P||_{L^{\infty}(\Omega_1)}\leq C\,\,\,\,||\nabla P||_{L^{\infty}(\Omega_1)}\leq C
\end{equation}
\end{lem}

\subsection*{Acknowledgments}
The author would like to thank Mike Cullen for suggesting me this problem, and Mikhail Feldman for helpful discussions and suggestions.

\end{document}